\documentclass[12pt, a4paper]{article}

\usepackage{amsmath}
\usepackage{amsfonts}
\usepackage{amssymb}
\usepackage[francais,english]{babel}

\def\english{\selectlanguage{english}}

\providecommand\mathbb{\bf}
\newcommand\R{{\mathbb R}}
\newcommand\N{{\mathbb N}}
\newcommand\Z{{\mathbb Z}}

\newtheorem{thm}{Theorem}[section]
\newtheorem{lemma}{Lemma}[section]

\newtheorem{pro}{Proposition}[section]
\newtheorem{defi}{Definition}[section]

\newtheorem{remark}{Remark}[section]

\newcounter{Remark}

\renewcommand\theRemark{\arabic{Remark}}

\newcounter{steps}
\newenvironment{proof}[1][]{%
\par\medbreak\setcounter{steps}{0}
{\noindent\bfseries Proof#1. }} {\hfill\fbox{\ }\medbreak}

\newcounter{substeps}[steps]

\newcommand{\supp}[0]{
\mathrm{supp\;}}

\newcommand{\tes}[0]{
\theta (s)}

\newcommand{\teps}[0]{
\theta ^{\;\prime} (s)}

\newcommand{\tet}[0]{
\theta  ({t}/{\varepsilon} )}

\newcommand{\tept}[0]{
\theta ^ {\;\prime}  ( {t}/{\varepsilon} )}

\newcommand{\Xe}[0]{
X ^\varepsilon }

\newcommand{\Pe}[0]{
P ^\varepsilon }

\newcommand{\s}[0]{
t/\varepsilon}

\newcommand{\Xz}[0]{
X^0}

\newcommand{\Xo}[0]{
X^1}

\newcommand{\Pz}[0]{
P^0}

\newcommand{\Po}[0]{
P^1}

\newcommand{\Qz}[0]{
Q^0}

\newcommand{\intsxp}[1]{
\int _0 ^T \!\!\!\int_{\R^3}\!\int _{\R ^3} \!#1 \;\mathrm{d}p\mathrm{d}x\mathrm{d}s}

\newcommand{\intxp}[1]{
\int _{\R^3}\!\int _{\R ^3} \!\!#1 \;\mathrm{d}p\mathrm{d}x}

\newcommand{\intxq}[1]{
\int _{\R^3}\!\int _{\R ^3} \!\!#1 \;\mathrm{d}q\mathrm{d}x}

\newcommand{\lime}[0]{
\lim _{\varepsilon \searrow 0}}

\newcommand{\intp}[1]{
\int _{\R ^3} \!#1 \;\mathrm{d}p}

\newcommand{\intq}[1]{
\int _{\R ^3} \!#1 \;\mathrm{d}q}

\newcommand{\intx}[1]{
\int _{\R ^3} \!#1 \;\mathrm{d}x}

\newcommand{\eps}[0]{
\varepsilon}

\newcommand{\fe}[0]{
f ^\varepsilon}

\newcommand{\Fe}[0]{
F ^\varepsilon}

\newcommand{\fz}[0]{
f ^0}

\newcommand{\fo}[0]{
f ^1}

\newcommand{\Rin}[0]{
R ^{\mathrm{in}}}

\newcommand{\fin}[0]{
f ^{\mathrm{in}}}

\newcommand{\gin}[0]{
g ^{\mathrm{in}}}

\newcommand{\gz}[0]{
g ^0}

\newcommand{\Be}[0]{
B ^\varepsilon}

\newcommand{\Ee}[0]{
E ^ \varepsilon}

\newcommand{\Ae}[0]{
A ^ \varepsilon}

\newcommand{\Div}[0]{
\mathrm{div}_x}

\newcommand{\Divy}[0]{
\mathrm{div}_y}

\newcommand{\Curl}[0]{
\mathrm{curl}_x}

\newcommand{\ran}[0]{
\mathrm{Range\;}}

\newcommand{\ltpltxp}[0]{
L^2_{\#}(\R _s;L^2(\R^3 _x \times \R ^3_p))}

\newcommand{\lipltxp}[0]{
L^\infty _{\#}( \R_s ; L^2(\R ^3 _x  \times \R ^3 _p))}

\newcommand{\ltxp}[0]{
L^2(\R  ^3 \times \R ^3)}

\newcommand{\oc}[0]{
\omega _c}

\newcommand{\ave}[1]{
\left \langle #1 \right \rangle }

\newcommand{\pb}[0]{
(p \cdot b(x) )}

\newcommand{\pwb}[0]{
|p \wedge b(x)|}

\newcommand{\loloctlinfx}[0]{
L^1 _{\mathrm{loc}}(\R_+; L ^\infty (\R^3))}

\newcommand{\linftltxp}[0]{
L^\infty(\R_+; L^2 (\R ^3 _x \times \R ^3 _p))}

\newcommand{\linf}[0]{
L^\infty (\R^3)}

\newcommand{\lty}[0]{
L^2( \R ^m)}

\newcommand{\inty}[1]{
\int _{\R ^m} \!\!\!\!#1 \;\mathrm{d}y}

\newcommand{\limT}[0]{
\lim _{T \to  +\infty }}

\newcommand{\limn}[0]{
\lim _{n \to  +\infty }}

\newcommand{\liy}[0]{
L^\infty( \R ^m)}

\baselinestretch\renewcommand{\baselinestretch}{1.5}

\begin{document}
\english

\title{Transport of charged particles under fast oscillating magnetic fields}

\author{
Mihai Bostan
\thanks{Laboratoire de
Math\'ematiques de Besan{\c c}on, UMR CNRS 6623, Universit\'e de
Franche-Comt\'e, 16 route de Gray, 25030 Besan{\c c}on  Cedex
France. E-mail : {\tt mbostan@univ-fcomte.fr}} 
}

\date{ (\today)}

\maketitle

\begin{abstract}
The energy production through thermo-nuclear fusion requires the confinement of the plasma into a bounded domain. In most of the cases, such configurations are obtained by using strong magnetic fields. Several models exist for describing the evolution of a strongly magnetized plasma, {\it i.e.,} guiding-center approximation, finite Larmor radius regime, etc. The topic of this paper concerns a different approach leading to plasma confinement. More exactly we are interested in mathematical models with fast oscillating magnetic fields. We provide rigorous derivations for this kind of models and analyze their properties. 
\end{abstract}



\section{Introduction}
\label{Intro}
\indent

Motivated by the energy production through thermo-nuclear fusion, many research programs concern plasma confinement models. It is well known that good confinement properties are obtained under strong magnetic fields $\Be = {\cal O } ( 1/\eps)$ with $\eps >0$ a small parameter. Using the kinetic description and neglecting the particle collisions lead to the Vlasov equation
\begin{equation}
\label{Equ1} \partial _t \fe + \frac{p}{m} \cdot \nabla _x \fe + e \left ( E^\eps (t,x) + \frac{p}{m} \wedge B^\eps (t,x) \right ) \cdot \nabla _p \fe = 0,\;\;(t,x,p) \in \R_+ \times \R ^3 \times \R ^3
\end{equation}
with the initial condition
\begin{equation}
\label{EquIC} 
\fe (0,x,p) = \fin (x,p),\;\;(x,p) \in \R ^3 \times \R ^3.
\end{equation}
Here $\fe = \fe (t,x,p ) \geq 0$ is the distribution function of the particles in the position-momentum phase space $(x,p) \in \R ^3 \times \R ^3$, $m$ is the particle mass, $e$ is the particle charge and $(E ^\eps, B^\eps)$ stands for the electro-magnetic field.

Standard configurations ensuring confinement are those obtained by applying strong magnetic fields. For example, assuming that the electric field derives from a given potential $E = - \nabla _x \phi$ and the magnetic field is stationary, divergence free 
$$
\Be (x) = \frac{B(x)}{\eps}b(x),\;\;\Div (Bb) = 0,\;\;0 < \eps <<1
$$
for some scalar positive function $B(x)$ and some field of unitary vectors $b(x)$, lead to the Vlasov equation 
\begin{equation}
\label{EquNew1}
\partial _t \fe + \frac{p}{m} \cdot \nabla _x \fe + \left ( e E (t,x) + \frac{\oc (x)}{\eps} p \wedge b (x) \right ) \cdot \nabla _p \fe  = 0,\;\;\oc (x) = \frac{eB(x)}{m}
\end{equation}
whose limit as $\eps \searrow 0$ is known as the guiding-center approximation. The particles rotate around the magnetic lines and the radius of this circular motion, which is called the Larmor radius $\rho _L$, is proportional to the inverse of the magnetic field. Therefore when the magnetic field is strong, the typical Larmor radius vanishes and the particles remain confined along the magnetic lines. But the frequency of these rotations, which is called the cyclotronic frequency, is proportional with the magnetic field. Consequently, high magnetic fields introduce small time scales, since the cyclotronic period is much smaller than the observation time. Clearly, the transport equation \eqref{EquNew1} involves multiple scales: fast motion around the magnetic lines driven by the Laplace force in $\frac{\oc (x)}{\eps} (p \wedge b ) \cdot \nabla _p$ and slow motion corresponding to the advection $\frac{p}{m} \cdot \nabla _x + eE \cdot \nabla _p$. 

From the numerical point of view, the efficient resolution of \eqref{EquNew1} requires multiple scale analysis \cite{BogMit61} or homogenization techniques. It is also possible to appeal to Lagrangian and Hamiltonian methods \cite{BriHam07}. For a unified treatment of the main physical ideas and theoretical methods that have emerged on magnetic plasma confinement we refer to \cite{HazMei03}.

The guiding-center approximation for
the Vlasov-Maxwell system was studied in \cite{BosSonPart1} by the
modulated energy method. The case of three dimensional general magnetic shapes \eqref{EquNew1} has been studied in \cite{BosGuidCent3D}, using a general method, based on ergodicity, introduced in \cite{BosTraSin}. It was proved in \cite{BosGuidCent3D} that the limit density $f = \lim _{\eps \searrow 0} \fe$ satisfies the Vlasov equation
\[
\partial _t f +  b(x) \otimes b(x) \frac{p}{m}  \cdot \nabla _x f + \left ( e b(x) \otimes b(x) E + \omega  (x,p) \;\tilde{p}
\right ) \cdot \nabla _p f = 0
\]
where for any $(x,p)$ with $p \wedge b(x) \neq 0 $ the symbol
$\tilde{p} $ stands for the orthogonal momentum to $p$,
contained in the plane determined by $b(x)$ and $p$, and such that its coordinate along $b(x)$ is positive, that means
\[
\tilde{p} = \pwb{} \;b(x) - \pb{} \;\frac{b(x) \wedge ( p \wedge b(x))}{\pwb{}}
\]
and the frequency $\omega (x,p)$ is given by
$$
\omega  (x,p) = \frac{\pwb{}}{2m} \;\Div b - \frac{\pb{}}{m} \;\left ( \frac{\partial b }{\partial x } b(x) \cdot \frac{p}{\pwb{}} \right ),\;\;p \wedge b(x) \neq 0.
$$
The analysis of the Vlasov or
Vlasov-Poisson equations with large external magnetic field have
been carried out in \cite{FreSon98}, \cite{GolSai99}, \cite{Bre00},
\cite{FreSon01}, \cite{GolSai03}. The numerical approximation of the
gyrokinetic models has been performed in \cite{GraBruBer06} using
semi-Lagrangian schemes. Other methods are based on the water bag
representation of the distribution function 
\cite{MorGraBes08}.

Notice that configurations with large magnetic field amplitude require huge energy since the magnetic energy is quadratic with respect to $|\Be| = B/\eps$. 

We investigate here models with fast oscillating magnetic fields
\begin{equation}
\label{Equ3}
\Be (t,x) = \theta (t/\eps)B(x)b(x),\;\;0 < \eps <<1
\end{equation}
where $\theta = \theta (s)$ is a given $T$ periodic profile of class $C^1$. The magnetic energy dissipated in this case is much lower than for the guiding-center approximation and remains of order of $|B|^2$. Therefore such models will be more interesting for real life applications, provided they still have good confinement properties. At this stage we neglect the gradient and curvature effects of the magnetic field, assuming that
$$
\Be = \Be (t) = \tet{}(0,0,B)
$$
for some constant $B>0$. The general model including gradient and curvature effects will be discussed in Section \ref{3DSetting}. The vector potential corresponding to this magnetic field, {\it i.e.,} satisfying $\Be = \Curl \Ae, \Div \Ae = 0$ is given by
$$
\Ae (t,x) = - \frac{B}{2}\; \tet{} ^\perp x,\;\;^\perp x = (x_2, - x_1, 0).
$$
Decomposing the electric field into gradient and rotational parts
\[
E ^\eps = - \nabla _x \phi + \Curl \psi ^\eps
\]
we deduce, by Faraday's law $\partial _t \Be + \Curl E^\eps = 0$ that
\[
\Curl ( \partial _t \Ae + \Curl \psi ^\eps ) = 0,\;\;\Div ( \partial _t \Ae + \Curl \psi ^\eps ) = 0
\]
and therefore the electric field induced by the time fluctuations of the magnetic field is
$$
\Curl \psi ^\eps   = - \partial _t \Ae  = \frac{B}{2 \eps }\; \tept{} ^\perp x.
$$
The Vlasov equation becomes, with the notations $^\perp p = (p_2, - p_1, 0)$ and $E = - \nabla _x \phi$ 
\begin{equation}
\label{Equ4} \partial _t \fe + \frac{p}{m} \cdot \nabla _x \fe + \left (eE (t,x) + \frac{m \oc}{2 \eps } \tept{} ^\perp x + \oc \; \tet{} ^\perp  p\right ) \cdot \nabla _p \fe = 0.
\end{equation}
Here $E = - \nabla _x \phi $ is a given irrotational electric field or can be determined eventually by solving the Poisson equation 
\begin{equation}
\label{Equ2}
\Div E (t)= - \Delta _x \phi  (t) = \frac{e}{\eps _0} \left \{   \intp{\fe (t,x,p)} - n_0 (x) \right \},\;\;t \in \R_+, \;\;x \in \R ^3.
\end{equation}
The concentration $n_0(x)$ corresponds to a neutralizing background of charged particles of opposite sign and $\eps _0$ is the electric permittivity of the vacuum. 

Our paper is organized as follows. The main results are presented in Section \ref{PreMainResult}. In Section \ref{AveOpe} we introduce the mathematical tools that we need for our analysis. It mainly concerns the average operator with respect to characteristic flows. We discuss its main properties as range characterization, Poincar\'e and Sobolev inequalities. Some commutation results are established in Section \ref{CommAveDer}. Section \ref{LimMod} is devoted to the derivation of the limit model, which follows in a natural way by appealing to the average operator  introduced before. We establish the conservation of the total energy and justify the confinement properties for such a model. The asymptotic behaviour towards this limit model is analyzed in Section \ref{AsyBeh}. The general three dimensional setting is investigated in the last section.

\section{Presentation of the model and main results}
\label{PreMainResult}

The Vlasov equation \eqref{Equ4} reduces to the characteristic system
\begin{equation}
\label{Equ5} 
\frac{d\Xe}{dt} = \frac{\Pe (t)}{m},\;\; \frac{d\Pe}{dt} =   e E(t,\Xe(t)) + \frac{m \oc}{2\eps}  \tept{} \;^\perp \Xe (t) + \oc \tet{}\;^\perp \Pe (t).
\end{equation}
It is convenient to introduce the fast variable $s = \s$ and the standard ansatz
$$
\Xe (t) = \Xz (t,\s) + \eps \Xo (t,\s) + ...,\;\;\Pe (t) = \Pz (t,\s) + \eps \Po (t,\s) + ...
$$
(here all dependences with respect to the fast variable $s$ are supposed $T$ periodic, as the profile $\theta = \tes{}$) leading to
$$
\partial _t \Xz + \frac{1}{\eps} \partial _s \Xz + \eps ( \partial _t \Xo + \frac{1}{\eps} \partial _s \Xo ) + ... = \frac{\Pz}{m} + \eps \frac{\Po}{m} +...
$$
and
\begin{eqnarray}
\partial _t \Pz + \frac{1}{\eps} \partial _s \Pz + \eps ( \partial _t \Po + \frac{1}{\eps} \partial _s \Po ) + ... & =& e E(t,\Xz + \eps \Xo +...) \nonumber \\
& + & \frac{m \oc}{2\eps}  \tept{} \;^\perp (\Xz + \eps \Xo + ...) \nonumber \\
& +&  \oc  \tet\;^\perp (\Pz + \eps \Po + ...).\nonumber
\end{eqnarray}
At least formally one gets the equations
\begin{equation}
\label{Equ6} \partial _s \Xz = 0,\;\;\partial _s \Pz = \frac{m\oc}{2}\teps{} \;^\perp \Xz
\end{equation}
at the lowest order $\eps ^{-1}$ and 
\begin{equation}
\label{Equ7} \partial _t \Xz + \partial _s \Xo = \frac{\Pz}{m},\;\; \partial _t \Pz + \partial _s \Po = e E(t,\Xz) + \frac{m\oc}{2} \teps{}\;^\perp \Xo + \oc \tes{} \;^\perp  \Pz
\end{equation}
at the next order $\eps ^0$. It follows that the quantities 
$$
\Xz,\;\; \Qz = \Pz - \frac{m\oc}{2} \tes{} \;^\perp \Xz
$$
depend only on $t$. Therefore, in order to obtain the characteristic equations satisfied by the leading order terms $(\Xz, \Pz)$ we write the equations \eqref{Equ7} in terms of $(\Xz, \Qz)$ and eliminate $(\Xo, \Po)$ by averaging with respect to the fast variable $s$ over one period. The first equation in \eqref{Equ7} becomes
$$
\partial _t \Xz + \partial _s \Xo = \frac{\Qz}{m} + \frac{\oc{}}{2} \tes{} {^\perp \Xz}
$$
and therefore averaging with respect to $s$ yields 
\begin{equation}
\label{Equ8}
\frac{d\Xz}{dt} = \frac{\Qz(t)}{m} + \frac{\oc{}}{2} \ave{\theta} \;^\perp \Xz (t),\;\;\ave{\theta} = \frac{1}{T}\int _0 ^T \theta (s)\;\mathrm{d}s.
\end{equation}
Similarly, the second equation in \eqref{Equ7} implies
\begin{eqnarray}
\partial _t \left ( \Qz + \frac{m \oc}{2} \tes{} \;^\perp \Xz   \right ) + \partial _s \Po & = & eE(t,\Xz) + \frac{m\oc}{2} \partial _s \{ \theta \;^\perp \Xo \} - \frac{m\oc}{2} \tes{} \partial _s \;^\perp \Xo \nonumber \\
& + & \oc \tes{} \;^\perp \{\Qz + \frac{m\oc}{2} \tes{} \;^\perp \Xz   \} \nonumber \\
& = & eE(t,\Xz) + \frac{m\oc}{2} \partial _s \{ \theta \;^\perp \Xo \} \nonumber \\
& - & \frac{m\oc}{2} \tes{} \;^\perp \left \{\frac{\Qz}{m} + \frac{\oc}{2} \tes{} \;^\perp \Xz - \partial _t \Xz \right \} \nonumber \\
& + & \oc \tes{} \;^\perp \Qz + \frac{m\oc ^2 }{2} \theta ^2 (s) \;^{\perp \perp}\Xz.\nonumber
\end{eqnarray}
Finally one gets
\begin{eqnarray}
\partial _t \Qz  + \partial _s \Po & = & e E(t,\Xz) + \frac{m\oc}{2} \partial _s \{ \theta \;^\perp \Xo\} + \frac{\oc}{2}\tes{}\;^\perp \Qz + \frac{m\oc ^2}{4} \theta ^2 (s) \;^{\perp \perp} \Xz \nonumber
\end{eqnarray}
and therefore, averaging with respect to $s$ yields
\begin{equation}
\label{Equ9}
\frac{d\Qz}{dt} =e E(t,\Xz(t)) + \frac{\oc}{2} \ave{\theta} \;^\perp \Qz (t) + \frac{m\oc ^2}{4} \ave{\theta ^2} \;^{\perp \;\perp} \Xz (t),\;\;\ave{\theta ^2} = \frac{1}{T}\int_0 ^T \theta ^2 (s) \;\mathrm{d}s.
\end{equation}
We associate to the characteristic equations \eqref{Equ8}, \eqref{Equ9} the transport equation
\begin{equation}
\label{Equ10}
\partial _t g^0 + \left ( \frac{q}{m} + \frac{\omega _c}{2} \ave{\theta} \;^\perp x \right ) \cdot \nabla _x g^0 + \left ( e E (t,x) + \frac{\omega _c}{2} \ave{\theta} \;^\perp q + \frac{m\oc ^2}{4}\ave{\theta ^2}  \;^{\perp \perp} x \right ) \cdot \nabla _q g^0 = 0.
\end{equation}
Since we have 
\[
( \Xe (t), \Pe (t)) \approx ( \Xz(t,\s), \Pz (t,\s)) = \left (\Xz (t), \Qz (t) + \frac{m\oc}{2}\tet{} \;^\perp \Xz (t) \right )
\]
and assuming that $\Xe (0) = x, \Pe (0) = p$ we can write
\[
f^\eps (t, \Xz (t,\s), \Pz (t,\s))  \approx  f ^ \eps (t, \Xe (t), \Pe (t)) = \fin (x,p)
\]
\[
g ^ 0 (t, \Xz(t,\s), \Pz(t,\s) - \frac{m\oc}{2} \tet{} \;^\perp \Xz (t,\s)) = g ^ 0  (t, \Xz (t), \Qz (t)) = \gin (x,q)
\]
and therefore we can expect that 
\[
\fe (t,x,p) \approx g ^ 0 (t,x,q = p - \frac{m\oc}{2} \tet{} \;^\perp x),\;\;\mbox{as } \eps \searrow 0
\]
provided that the initial conditions are well prepared. Introducing the density $\fz (t,s,x,p) = \gz (t, x, p - m \oc \tes{} /2  \;^\perp x)$ in the phase space $(s,x,p)$ we deduce that $\fe (t,x,p) \approx \fz (t, \s, x, p)$ as $\eps \searrow 0$. 

\begin{thm}
\label{MainResultWeak} Assume that $E \in \loloctlinfx{}$, $\fin \in \ltxp{}$. For any $\eps >0$ let $\fe \in \linftltxp{}$ be a weak solution of \eqref{Equ4}. Then there is a sequence $(\eps _n)_n$ converging to zero such that $(f ^{\eps _n})_n$ two-scale converges towards a weak solution of
\begin{eqnarray}
\label{EquNew3} & & \partial _t \fz + \left (\frac{p}{m} - \frac{\oc}{2} ( \tes{} - \ave{\theta})\;^\perp x    \right ) \cdot \nabla _x \fz   \\
& + & \left (eE(t,x) + \frac{\oc}{2} ( \tes{} + \ave{\theta}) \;^\perp p + \frac{m\oc ^2}{4} ( \ave{\theta ^2} - \theta ^2 (s)) \;^{\perp \perp} x    \right ) \cdot \nabla _p \fz = 0\nonumber
\end{eqnarray}
\begin{equation}
\label{EquNew4} \fz(0, s, x, p) = \fin \left (x, p - \frac{m\oc}{2} (\tes{} - \theta (0))\;^ \perp x     \right )\in \ker {\cal T}.
\end{equation}
\end{thm}
Consequently we have      
 to study the confinement properties of the limit model \eqref{Equ10} (or \eqref{EquNew3}). Indeed, such models lead to confinement. For convincing ourselves let us consider a particular case, that of vanishing electric potential $\phi = 0$.  The characteristic system for $(\Xz, \Qz)$ becomes
\[
\frac{d\Xz}{dt} = \frac{\Qz (t)}{m} + \frac{\oc}{2}\ave{\theta} \;^\perp \Xz (t),\;\;\frac{d\Qz}{dt} = \frac{\oc}{2} \ave{\theta} \;^\perp \Qz (t) + \frac{m\oc ^2}{4} \ave{\theta ^2} \;^{\perp \perp} \Xz (t)
\]
implying that 
\begin{equation}
\label{EquNew6}
\frac{d^2\Xz}{d t^2} - \frac{\oc ^2}{4} ( \ave{\theta ^2} - \ave{\theta} ^2) \;^{\perp \perp }\Xz (t) = \oc \ave{\theta} \frac{d \;^\perp \Xz}{dt}.
\end{equation}
Multiplying by $\displaystyle \frac{d\Xz}{dt}$ we obtain the conservation
$$
\frac{d}{dt} \left \{ \frac{1}{2}\left |\frac{d\Xz}{dt}    \right |^2 + \frac{\oc ^2}{4} (\ave{\theta ^2} - \ave{\theta}^2) \frac{1}{2}|^{\perp \perp}\Xz |^2 \right \} = 0.
$$
If $\theta$ is not a constant profile ({\it i.e.,} the magnetic field oscillates in time), then 
\[
\frac{\oc ^2}{4} ( \ave{\theta ^2} - \ave{\theta} ^2) = \frac{\oc ^2}{4} \ave{(\theta - \ave{\theta})^2}>0
\]
and clearly the projection of $\Xz (t)$ on the orthogonal directions with respect to the magnetic field oscillates around the magnetic lines. The oscillation frequencies can be computed explicitely in this particular case. Observe that the components $(X_1, X_2)$ satisfy
\[
X_j ^{(4)} + \frac{\oc ^2}{2} ( \ave{\theta ^2} + \ave{\theta} ^2 ) X_j ^{(2)} + \frac{\oc ^4 }{16} ( \ave{\theta ^2} - \ave{\theta}^2 ) ^2 X_j = 0.
\]
The roots of the characteristic polynomial are purely imaginary 
\[
\pm i \; \frac{\oc}{2} ( \sqrt{\ave{\theta ^2}} \pm \ave{\theta} )
\]
and therefore the oscillation frequencies in the plane $(x_1, x_2)$ are $\frac{\oc}{2} ( \sqrt{\ave{\theta ^2}} \pm \ave{\theta} )$. The plasma remains confined along the magnetic lines. Generally we establish the following result
\begin{thm}
\label{MainResult2} Assume that $\lambda \in C^1 (\R)$ is nonincreasing, nonnegative and vanishes on $[L,+\infty[$, for some $L >0$. Let the initial condition $\fin $ satisfy
\[
\fin (x,p) \leq \lambda ( \chi (x, p - m \oc \theta (0) /2\;^\perp x) + e \phi (0,x))
\]
where 
\[
\chi (x,q) =  \frac{1}{2m} \left | q + \frac{m \oc}{2} \ave{\theta} \;^\perp x  \right | ^2 + \frac{m \oc ^2}{4} ( \ave{\theta ^2} - \ave{\theta} ^2 ) \frac{|^\perp x|^2}{2}.
\]
If the electric potential $\phi \in C^1(\R_+ \times \R^3)$ satisfies
\begin{equation}
\label{EquNew5}
\lim _{|^\perp x| \to +\infty} \left \{e \phi (t, x) - \int _0 ^t \sup _{y \in \R^3} \{ e \partial _t \phi (s,y) \}\;\mathrm{d}s + \frac{m\oc ^2}{4} ( \ave{\theta ^2} - \ave{\theta} ^2 ) \frac{|^\perp x|^2}{2}    \right \} = + \infty
\end{equation}
uniformly with respect to $t \in \R_+, x_3 \in \R$, then there is a constant $R >0$ such that for any $t \in \R_+, s \in \R$ the solution of the problem \eqref{EquNew3}, \eqref{EquNew4} verifies
\[
\supp \fz (t, s, \cdot, \cdot) \subset \{ (x,p) \;:\;| ^\perp x | \leq R\}.
\]
\end{thm}
We also prove a strong convergence result
\begin{thm}
\label{MainResult1} Assume that $E \in L^1_{\mathrm{loc}}(\R_+; W^{2,\infty}(\R^3))$, $\partial _t E \in \loloctlinfx{}$, the initial condition $\fin$ has compact support and belongs to $ W^{2,\infty}(\R^3 \times \R^3)$. Let $\fz (t,s, x,p)$ be the solution of \eqref{EquNew3}, \eqref{EquNew4} and $(\fe)_\eps$ the solutions of the problems \eqref{Equ4}, \eqref{EquIC}. Then for any interval $[0,I] \subset \R_+$ there is a constant $C(I)$ such that 
\[
\|\fe(t,\cdot, \cdot) - \fz(t,\s, \cdot, \cdot ) \|_{\ltxp{}} \leq C(I) \;\eps,\;\;t \in [0,I], \;\eps >0.
\]
\end{thm}

\section{Average operator}
\label{AveOpe}
\indent

The previous considerations clearly show that the limit of the Vlasov equation with fast oscillating magnetic field deals with multi-scale techniques and homogenization arguments. For the rigorous derivation of the limit model \eqref{Equ10} we appeal to a slightly different method, based on Hilbert expansion at the density level
\begin{equation}
\label{Equ14} \fe (t,x,p) = \fz(t,\s, x,p) + \eps \fo (t, \s, x, p) + ...
\end{equation}
In this section we assume that $E = - \nabla _x \phi$ is a given electric field. Plugging this ansatz into \eqref{Equ4} leads to 
\begin{eqnarray}
& & \partial _t \fz + \frac{1}{\eps} \partial _s \fz + \eps \left (\partial _t \fo + \frac{1}{\eps} \partial _s \fo 
 \right ) + ... + \frac{p}{m} \cdot ( \nabla _x \fz + \eps \nabla _x \fo + ...) \nonumber \\
 & + & \left ( eE (t,x) + \frac{m\oc}{2 \eps} \tept{}\;^\perp x + \oc \tet{}\;^\perp p \right ) \cdot ( \nabla _p  \fz + \eps \nabla _p \fo + ...) = 0 
\end{eqnarray}
and we obtain formally 
\begin{equation}
\label{Equ15} \partial _s \fz + \frac{m\oc}{2}\teps{} \;^\perp x \cdot \nabla _p \fz = 0
\end{equation}
at the lowest order $\eps ^{-1}$ and
\begin{equation}
\label{Equ16} \partial _t \fz + \frac{p}{m} \cdot \nabla _x \fz  + (eE (t,x) + \oc \tes{} \;^\perp p ) \cdot \nabla _p \fz + \partial _s \fo + \frac{m\oc}{2}\teps{} \;^\perp x \cdot \nabla _p \fo = 0
\end{equation}
at the next order $\eps ^0$. The following operator will play a crucial role in our analysis 
\begin{equation}
\label{Equ17}
{\cal T} u = \mathrm{div}_{(s,p)} \left \{u \left ( 1, \frac{m\oc}{2}\teps{} \;^\perp x\right ) \right \}
\end{equation}
with domain 
\[
D({\cal T}) = \{u   \in \ltpltxp{}:\mathrm{div}_{(s,p)} \left \{u \left ( 1, \frac{m\oc}{2}\teps{} \;^\perp x\right ) \right \} \in \ltpltxp{}\}
\]
with $L^2_{\#}(\R _s;X)$ the space of square integrable $T$ periodic functions $u : \R \to (X, \|\cdot \|_X)$, endowed with the norm 
\[
\left ( \int _0 ^T \|u(s) \|_X ^2 \;\mathrm{d}s \right ) ^{1/2}.
\]  
The notation $\|\cdot \|$ stands for the standard norm of $\ltpltxp{}$
$$
\|u \| = \left ( \intsxp{|u(s,x,p)|^2} \right ) ^{1/2}.
$$ 
We denote by $(S,X,P)= (S,X,P)(\tau;s,x,p)$ the characteristics of the first order differential operator $\partial _s + \frac{m \oc}{2} \teps{} \;^\perp x \cdot \nabla _p $
\begin{equation}
\label{Equ17Bis}
\frac{dS}{d\tau} = 1,\;\;\frac{dX}{d\tau} = 0,\;\;\frac{dP}{d\tau} = \frac{m\oc}{2}\theta ^{\;\prime}(S(\tau)) \;^\perp X(S(\tau))
\end{equation}
with the conditions 
\[
S(0;s,x,p) = s,\;\;X(0;s,x,p) = x,\;\;P(0;s,x,p) = p.
\]
It is easily seen that 
\begin{equation}
\label{Equ18} S(\tau;s,x,p) = s + \tau,\;\;X(\tau; s,x,p) = x,\;\;P(\tau;s,x,p) = p + \frac{m\oc}{2} ( \theta (s + \tau) - \theta (s)) \;^\perp x.
\end{equation}
Notice that $\{x, p - \frac{m \oc}{2} \tes{} \;^\perp x\}$ is a complete family of functional independent prime integrals of \eqref{Equ17Bis}. We introduce the average operator along the characteristic flow \eqref{Equ18} cf. \cite{BosTraSin}
\begin{eqnarray}
\label{EquBelong}
\ave{u} (s,x,p) & = &  \frac{1}{T}\int _0 ^T u (S(\tau;s,x,p), X(\tau;s,x,p), P(\tau;s,x,p)) \;\mathrm{d}\tau  \\
& = &  \frac{1}{T}\int _0 ^T u \left (s+ \tau, x, p + \frac{m\oc}{2} ( \theta (s + \tau) - \tes{}) \;^\perp x \right )\;\mathrm{d}\tau \nonumber \\
& = & \frac{1}{T} \int _0 ^T u \left ( \tau, x, p - \frac{m\oc}{2} \tes{} \;^\perp x + \frac{m \oc}{2} \theta (\tau) \;^\perp x    \right ) \;\mathrm{d}\tau\nonumber 
\end{eqnarray}
for any function $u \in \ltpltxp{}$. 
\begin{pro}
\label{P1} The average operator is linear continuous. It coincides with the orthogonal projection on the kernel of ${\cal T}$ {\it i.e.,} 
\[
\ave{u} \in \ker {\cal T}\;:\; \intsxp{(u - \ave{u}) \varphi } = 0,\;\;\forall \; \varphi \in \ker {\cal T}.
\]
\end{pro}
\begin{proof}
For any function $u \in \ltpltxp{}$ we have
\[
|\ave{u}|^2 (s,x,p) \leq \frac{1}{T} \int _0 ^T |u|^2 \left (s+\tau, x, p + \frac{m \oc}{2} ( \theta (s + \tau) - \tes{}) \;^\perp x    \right ) \;\mathrm{d}\tau
\]
implying that 
\begin{eqnarray}
\intxp{|\ave{u}|^2 (s)} & \leq & \frac{1}{T} \int _0 ^T \!\!\!\intxp{|u|^2 \left (s+ \tau, x, p + \frac{m\oc}{2} ( \theta (s + \tau) - \tes{} ) \;^\perp x \right ) }\mathrm{d}\tau \nonumber \\
& = & \frac{1}{T} \int _0 ^T \!\!\!\intxp{|u|^2 (s+ \tau, x, p)}\mathrm{d}\tau \nonumber \\
& = & \frac{1}{T}\| u \|^2.
\end{eqnarray}
Therefore we have 
\[
\|\ave{u} \| \leq \|u \|,\;\;\forall \; u \in \ltpltxp{}
\]
saying that $\ave{\cdot} \in {\cal L}(\ltpltxp{}, \ltpltxp{})$. It is well known that the kernel of ${\cal T}$ is given by the functions invariant along the characteristics \eqref{Equ18}
\[
\ker {\cal T} = \left \{ u \in \ltpltxp{}\;:\; \exists \;v  \mbox{ such that } u(s,x,p) = v\left ( x, p - \frac{m \oc}{2} \tes{} \;^\perp x \right ) \right \}.
\]
Clearly $\ave{u}$ depends only on $x$ and $p - \frac{m \oc}{2} \tes{} \;^\perp x$, cf. \eqref{EquBelong}, and thus $\ave{u}$ belongs to $\ker {\cal T}$.
Pick a function $\varphi \in \ker {\cal T}$ {\it i.e.,}
$$
\exists \; \psi \;:\; \varphi (s,x,p) = \psi \left (x, p - \frac{m \oc}{2} \tes{} \;^\perp x \right ) 
$$
and let us compute $I = \intsxp{\;(u - \ave{u}) \varphi }$. Using the change of coordinates $q = p - \displaystyle \frac{m \oc}{2} \tes{} \;^\perp x$, for fixed $(s,x)$, one gets
\begin{eqnarray}
I & = &  \int_0 ^T \!\!\!\int _{\R ^3} \!\int _{\R ^3} \!(u - \ave{u})\left ( s, x, q + \frac{m \oc}{2} \tes{} \;^\perp x \right ) \psi (x,q) \;\mathrm{d}q\mathrm{d}x\mathrm{d}s \nonumber \\
& = & \int _{\R ^3} \!\int _{\R ^3} \! \psi (x,q) \left \{ \int _0 ^T u \left (s, x, q + \frac{m \oc}{2} \tes{} \;^\perp x   \right )\mathrm{d}s - T \ave{u}   \right \}\;\mathrm{d}q\mathrm{d}x \nonumber \\
& = & 0 \nonumber 
\end{eqnarray}
and therefore $\ave{u} = \mathrm{Proj}_{\ker {\cal T}} u$ for any $u \in \ltpltxp{}$. In particular $\ave{u} = u$ for any $u \in \ker {\cal T}$.
\end{proof}
We investigate now the solvability of the equation ${\cal T}u = v$. We have a simple characterization in terms of the kernel of the average operator. Notice that if $v = {\cal T}u \in \ran{} {\cal T}$ we have for any $\varphi \in \ker {\cal T}$
\[
\intsxp{(v - 0) \varphi } = \intsxp{{\cal T} u \;\varphi } = - \intsxp{u \;{\cal T}\varphi } = 0
\]
saying by Proposition \ref{P1} that $\ave{v} = 0$. Moreover we have
\begin{pro}
\label{P2}
The restriction of ${\cal T}$ to $\ker \ave{\cdot}$ is one to one map onto $\ker \ave{\cdot}$. Its inverse belongs to ${\cal L }(\ker \ave{\cdot}, \ker \ave{\cdot} )$ and we have the Poincar\'e inequality
\[
\|u \| \leq T \|{\cal T} u \|\;\;\mbox{for any } u \in D({\cal T}) \cap \ker \ave{\cdot}.
\]
\end{pro}
\begin{proof}
We already know that $\ran{} {\cal T} \subset  \ker \ave{\cdot}$. Assume now that $u \in D({\cal T}) \cap \ker \ave{\cdot}$ such that ${\cal T } u = 0$. Since $\ave{\cdot} = \mathrm{Proj}_{\ker {\cal T}}$ we have $u = \ave{u} = 0$ saying that ${\cal T}|_{\ker \ave{\cdot}}$ is injective. Consider now $v \in \ker \ave{\cdot}$ and let us prove that there is $u \in \ker \ave{\cdot} \cap D ({\cal T})$ such that ${\cal T }u = v$. For any $\alpha >0$ there is a unique $u _\alpha  \in D({\cal T})$ such that  
\begin{equation}
\label{Equ21}
\alpha \; u _\alpha + {\cal T}u _\alpha = v.
\end{equation}
Indeed it is easily seen that the solutions $(u_\alpha)_{\alpha >0}$ are given by
$$
u _\alpha (s,x, p) = \int _{\R_-} e ^{\alpha \tau} v\left (s+\tau, x, p - \frac{m\oc}{2} \tes{} \;^\perp x + \frac{m \oc}{2} \theta (s + \tau) \;^\perp x \right )\;\mathrm{d}\tau.
$$
Applying the average operator to (\ref{Equ21}) yields $\ave{u _\alpha} = 0$ for any $\alpha >0$. We introduce the function 
\[
V(\tau;s,x,p) = \int _\tau ^0 v\left (s+r, x, p - \frac{m\oc}{2} \tes{} \;^\perp x + \frac{m\oc}{2} \theta (s+r) \;^\perp x \right )\;\mathrm{d}r.
\] 
Notice that for any fixed $(s,x,p)$ the function $\tau \to V(\tau;s,x,p)$ is $T$ periodic, because $\ave{v} = 0$ and thus $\|V(\tau)\| \leq T\|v\|$ for any $\tau \in \R$. Integrating by parts we obtain
$$
u _\alpha (s,x,p) = - \int _{\R_-} e ^{\alpha \tau} \partial _\tau V \;\mathrm{d}\tau = \int _{\R _-} \alpha e ^{\alpha \tau} V(\tau;s,x,p) \;\mathrm{d}s
$$
implying that 
\[
\|u _\alpha \| \leq \int _{\R _-} \alpha e ^{\alpha \tau} \|V(\tau) \|\mathrm{d}\tau \leq T \|v \|.
\]
Extracting a sequence $(\alpha _n )_n$ such that $\lim _{n \to +\infty} \alpha _n = 0$, $\lim _{n \to +\infty} u _{\alpha _n} = u$ weakly in $\ltpltxp{}$ we deduce easily that
$$
u \in D({\cal T}),\;\;{\cal T} u = v,\;\;\ave{u} = 0,\;\;\|u \| \leq T \|v\|
$$
saying that $\left ( {\cal T} |_{\ker \ave{\cdot}} \right ) ^{-1}$ is bounded linear operator and $\|\left ( {\cal T} |_{\ker \ave{\cdot}} \right ) ^{-1}\|_{{\cal L} (\ker \ave{\cdot},\ker \ave{\cdot})} \leq T$.
\end{proof}
\begin{remark}
\label{CommSupp} Notice that ${\cal T}^{-1}$ leaves invariant the set of zero average functions, with compact support. Indeed, if $v \in \ker \ave{\cdot}$ has compact support, let us say $\supp v \subset \{(s,x,p) : |x|\leq R, |p| \leq R\}$ for some $R >0$, it is easily seen that for any $\alpha >0$, the function $u _\alpha $ in \eqref{Equ21} has compact support (uniformly with respect to $\alpha$)
\[
\supp u _\alpha  \subset \{ (s,x,p) \;:\; |x| \leq R,\;\; |p| \leq ( 1 + m |\oc| \; \|\theta\|_{L^\infty} ) R\}
\]
and therefore the weak limit $u = \lim _{\alpha \searrow 0 } u _\alpha $ has compact support. 
\end{remark}
Notice that we have the orthogonal decomposition of $\ltpltxp{}$ into invariant functions along the characteristics (\ref{Equ17Bis}) and zero average functions $u = \ave{u} + ( u - \ave{u})$, since by Proposition \ref{P1} we have
\[
\intsxp{(u - \ave{u}) \ave{u} } = 0.
\]
We end this section with the following Sobolev inequality
\begin{pro}
\label{P3} There is a constant $C = C(T)$ such that for any function $u\in D({\cal T})$ we have
\[
\|u \|_{\lipltxp{}} \leq C(T) ( \|u \| + \|{\cal T}u \|).
\]
In particular for any function $u \in D({\cal T}) \cap \ker \ave{\cdot}$ we have 
\[
\|u \|_{\lipltxp{}} \leq C(T) ( 1 + T) \|{\cal T}u\|.
\]
\end{pro}
\begin{proof}
Without loss of generality we can assume that the function $u$ is smooth (the general case follows by standard density arguments). For any $s \in \R$ and $t \in [s-T, s]$ we have
\[
\frac{d}{ds}\left \{ u \left (s, x, p - \frac{m\oc}{2} ( \theta(t)- \tes{}) \;^\perp x   \right )    \right \} = ({\cal T}u ) \left (s, x, p - \frac{m\oc}{2} ( \theta (t) - \tes{}) \;^\perp x    \right ).
\]
After integration one gets
\[
u\left (s, x, p - \frac{m\oc}{2} ( \theta(t)- \tes{}) \;^\perp x   \right )  = u(t,x,p) + \int _t ^s {\cal T}u \left (\tau, x, p - \frac{m\oc}{2} ( \theta(t)- \theta(\tau)) \;^\perp x   \right ) 
\]
implying that 
\[
\|u (s,\cdot, \cdot)\|_{\ltxp{}} ^2 \leq 2 \|u (t,\cdot, \cdot) \|^2 _{\ltxp{}} + 2T \|{\cal T}u \|^2.
\]
Averaging with respect to $t$ over the period $[s-T, s]$ yields
\[
\|u \|^2 _{\lipltxp{}} \leq \frac{2}{T} \|u \|^2 + 2T \|{\cal T}u \|^2
\]
and the first statement follows with $C(T) = \max \{\sqrt{2/T}, \sqrt{2T} \}$. Moreover, if $u \in D({\cal T}) \cap \ker \ave{\cdot}$ we know by Proposition \ref{P2} that $\|u\| \leq T \|{\cal T}u\|$ and therefore 
\[
\|u\|_{\lipltxp{}} \leq C(T) \{T \|{\cal T}u \| + \|{\cal T}u \|) = C(T) (1 + T) \|{\cal T}u\|.
\]
\end{proof}

\section{Commutation properties of the average with respect to derivations}
\label{CommAveDer}

We have seen that ${\cal T }^{-1}$ restricted to zero average functions is linear and continuous. In view of further regularity  that we need for the asymptotic analysis of \eqref{Equ4} we investigate now the action of ${\cal T}^{-1}$ and $\ave{\cdot}$ over subspaces of smooth functions. We formulate our statements in the general framework of a characteristic flow associated to smooth, divergence free fields. Let $b : \R ^m \to \R^m$ be a field satisfying
\[
b \in W^{1,\infty}_{\mathrm{loc}} (\R ^m),\;\;\Divy  b = 0
\]
and the growth condition
\[
\exists \;C >0 \;:\; |b(y) | \leq C ( 1 + |y|),\;\;y
\in \R ^m.
\]
Under the above hypotheses the characteristic flow $Y = Y(s;y)$ is
well defined
\[
\frac{dY}{ds} = b(Y(s;y)),\;\;(s,y) \in \R \times \R
^m
\]
\[
Y(0;y) = y,\;\; y \in \R ^m,
\]
and has the regularity $Y \in W^{1,\infty}_{\mathrm{loc}} ( \R
\times \R ^m)$. Since the field is divergence free, we deduce by Liouville's theorem that for any $s \in
\R$, the map $y \to Y(s;y)$ is measure preserving. Notice that we don't make any periodicity assumption on the flow $Y$. As usual the notation $b \cdot \nabla _y$ stands for the first order differential operator with domain
\[
\mathrm{D}(b \cdot \nabla _y) = \{ u \in \lty{}\;:\; \Divy (u (y) b(y))  \in \lty{}
\}
\]
which maps any $u \in \mathrm{D}(b \cdot \nabla _y)$ to the function $\Divy (u (y) b(y)) $. The kernel of this operator is given by $L^2$ functions which are constant along the flow $Y$
\[
\ker (b \cdot \nabla _y  ) = \{ u \in \lty{}\;:\; u (Y(s;y)) = u(y),\;s\in
\R,\;\mathrm{a.e.}\; y \in \R^m\}.
\]
It is easily seen that for any function $u \in \lty{}$ and any $T >0$
\[
\left \| \frac{1}{T} \int _0 ^T u (Y(s;\cdot))\;\mathrm{d}s    \right \| _{\lty{}} \leq \|u \| _{\lty{}}
\]
and therefore we can expect compactness properties for the family of averages along the flow $Y$ {\it i.e.,} $\{ \frac{1}{T} \int _0 ^T u (Y(s;\cdot))\;\mathrm{d}s , T >0   \}$. Indeed, the
mean ergodic theorem, or von Neumann's ergodic theorem (see
\cite{ReedSimon}, pp. 57) allows us to construct the average operator along the flow $Y$.
\begin{pro}
\label{LargeTStrong}
\label{GenAveOpe} For any function $u \in \lty{}$ the averages $
\frac{1}{T} \int _0 ^T u (Y(s;\cdot))\;\mathrm{d}s$ converge strongly in $\lty{}$, when $T \to +\infty$, towards some function denoted $\ave{u} \in \lty{}$. The map $\ave{\cdot} : \lty{} \to \lty{}$ is linear, continuous and coincides with the orthogonal projection on $ \ker (b \cdot \nabla _y )$
\[
\ave{u} \in \ker ( b \cdot \nabla _y ) \;:\; \inty{(u(y) - \ave{u}(y)) \varphi (y) } = 0,\;\forall \;\varphi \in \ker ( b \cdot \nabla _y ).
\] 
\end{pro}
It is easily seen that $\ran ( b \cdot \nabla _y) \subset \ker \ave{\cdot}$. Actually it is shown in \cite{BosTraSin} that $\overline{\ran  ( b \cdot \nabla _y)} = \ker \ave{\cdot}$. We are searching now for derivations commuting with the average operator. We prove that the average operator is commuting with any derivation $c \cdot \nabla _y$ associated to a field $c$ in involution with $b$. 
We recall here the following basic
results concerning derivation operators along fields in $\R
^m$. For any $\xi = (\xi _1(y),...,\xi_m (y))$, where $y \in \R
^m$, we denote by $L_\xi $ the operator $\xi \cdot \nabla _y$. A
direct computation shows that for any smooth fields $\xi, \eta $,
the commutator between $L_\xi, L_\eta$ is still a first order
operator, given by
$
[L_\xi, L_\eta] := L_\xi L_\eta - L_\eta L_\xi = L_\chi,
$
where $\chi $ is the Poisson bracket of $\xi$ and $\eta$
$$
\chi = [\xi, \eta],\;\;[\xi, \eta]_i = (\xi \cdot \nabla _y ) \eta
_i -(\eta \cdot \nabla _y ) \xi _i = L_\xi (\eta _i) - L_\eta (\xi
_i),\;\;i \in \{1,...,m\}.
$$
It is well known (see \cite{Arnold}, pp. 93) that $L_\xi, L_\eta$
commute (or equivalently the Poisson bracket $[\xi, \eta]$
vanishes) iff the flows corresponding to $\xi, \eta$, let say
$Z_1, Z_2$, commute
$$
Z_1(s_1;Z_2(s_2;y)) = Z_2(s_2;Z_1(s_1;y)),\;\;s_1, s_2 \in
\R,\;\;y \in \R ^m.
$$
Consider a smooth field $c$ in involution with $b$ and having
bounded divergence
$$
c \in W^{1,\infty}_{\mathrm{loc}} (\R ^m),\;\;\Divy c \in
\liy{},\;\;[c,b] = 0
$$
and let us denote by $Z$ the flow associated to $c$ (we assume
that $Z$ is well defined for any $(s,y) \in \R \times \R ^m$). We claim that the following commutation property holds true.

\begin{pro}
\label{Invariance} Assume that $c$ is a smooth field in involution
with $b$, with bounded divergence and well defined flow. Then the average operator  commutes
with the translations along the flow of $c$
$$
\ave{u \circ Z(h;\cdot)} = \ave{u} \circ
Z(h;\cdot),\;\;u \in \lty{},\;\;h \in \R.
$$
\end{pro}
\begin{proof}
The commutation property of the flows $Y, Z$ and Proposition \ref{LargeTStrong} allow us to write the strong convergences in $\lty{}$
\begin{eqnarray}
\nonumber
\ave{u \circ Z(h;\cdot)} & = & \limT \frac{1}{T} \int _0 ^ T u \circ Z(h;Y(s;\cdot))\;\mathrm{d}s \\
& = & \limT \frac{1}{T} \int _0 ^ T u \circ Y(s;Z(h;\cdot))\;\mathrm{d}s \\
& = & \left (\limT \frac{1}{T} \int _0 ^ T u (Y(s;\cdot))\;\mathrm{d}s
\right ) \circ Z(h;\cdot) \\
& = & \ave{u} \circ Z(h;\cdot).
\end{eqnarray}
Notice that the third equality in the above computations follows by changing the variable along the flow $Z$ and by using the boundedness of $\mathrm{div}_y c$.
\end{proof}
We denote by $c \cdot \nabla _y$ the operator with the domain
\[
\mathrm{D}(c \cdot \nabla _y) = \{ u \in \lty{}\;:\; \Divy (u (y) c(y))  \in \lty{}
\}
\]
which maps any function $u \in \mathrm{D}(c \cdot \nabla _y)$ to the function $\Divy (uc) - u \;\Divy c$. 
\begin{pro}
\label{DerCom} Under the hypotheses of Proposition
\ref{Invariance}, assume that $u \in \mathrm{D}(c \cdot \nabla _y)$. Then $\ave{u} \in
\mathrm{D}(c \cdot \nabla _y)$  and $c \cdot \nabla _y \ave{u}  =
\ave{c \cdot \nabla _y u}$.
\end{pro}
\begin{proof}
For any $h \in \R^\star$ we have
\[
\frac{\ave{u} \circ Z(h;\cdot) - \ave{u}}{h} = \frac{\ave{u \circ Z(h;\cdot)} - \ave{u}}{h} = \ave{\frac{u \circ Z(h;\cdot) - u}{h}}.
\]
Since $u \in \mathrm{D}(c \cdot \nabla _y)$ we have the strong convergence $\lim _{h \to 0} (u \circ Z(h;\cdot) - u)/h = c \cdot \nabla _y u$ in $\lty{}$ and by the continuity of the average operator, we deduce that 
\[
\lim _{h \to 0} \frac{\ave{u} \circ Z(h;\cdot) - \ave{u}}{h} = \ave{c \cdot \nabla _y u}
\]
strongly in $\lty{}$, saying that $\ave{u} \in \mathrm{D}(c \cdot \nabla _y)$ and $c \cdot \nabla _y \ave{u} = \ave{c \cdot \nabla _y u}$.
\end{proof}
Similarly we can prove
\begin{pro}
\label{TimDer} Assume that $u \in W^{1,p}([0,T];\lty{})$ for some
$p \in (1,+\infty)$. Then the application $(t,y) \to
\ave{u(t,\cdot)}(y)$ belongs to $W^{1,p}([0,T];\lty{})$ and
we have $\partial _t \ave{u} = \ave{\partial _t u }$.
\end{pro}
We are ready now to study how the regularity propagates under the action of $(b \cdot \nabla _y )^{-1}$. We make the following assumption (Poincar\'e inequality)
\begin{equation}
\label{EquPoincareConstant}
\exists \;C_P >0\;\mbox{such that}\; \|u\|_{\lty{}} \leq C_P \|b \cdot \nabla _y u \|_{\lty{}},\;\forall \; u \in \mathrm{D}(b \cdot \nabla _y ) \cap \ker \ave{\cdot}
\end{equation}
meaning that the range of $b \cdot \nabla _y$ is closed {\it i.e.,} $\ran (b \cdot \nabla _y ) = \overline{\ran (b \cdot \nabla _y )} = \ker \ave{\cdot}$ and $b \cdot \nabla _y$ restricted to $\ker \ave{\cdot}$ is one to one map onto  $\ker \ave{\cdot}$ with bounded inverse. Notice that the above hypothesis is satisfied by the operator in \eqref{Equ17}, cf. Proposition \ref{P2}. 
\begin{pro}
\label{RegPropag} Under the hypotheses of Proposition \ref{Invariance} and \eqref{EquPoincareConstant} assume that $v \in \ker \ave{\cdot} \cap \mathrm{D}(c \cdot \nabla _y)$. Let us denote by $u$ the unique zero average solution of $b \cdot \nabla _y u = v$. Then $u \in \mathrm{D} (c \cdot \nabla _y)$ and $\|c \cdot \nabla _y u \|_{\lty{}} \leq C_P \|c \cdot \nabla _y v \|_{\lty{}}$.
\end{pro}
\begin{proof}
For any $h \in \R ^\star$ we have $u _h := u \circ Z(h;\cdot) \in \mathrm{D}(b \cdot \nabla _y)$ and $b \cdot \nabla _y u _h = (b \cdot \nabla _y u)_h = v_h$. Therefore we deduce that $b \cdot \nabla _y (u_h - u) = v_h - v$. Since the average is commuting with the translations along the flow of $c$ we have
\[
\ave{u_h - u} = \ave{u_h} - \ave{u} = \ave{u}_h - \ave{u} = 0 - 0 = 0
\]
and we can apply the Poincar\'e inequality \eqref{EquPoincareConstant}
\[
\|u _h - u\|_{\lty{}} \leq C_P \|v_h - v\|_{\lty{}},\;h \in \R ^\star.
\]
Therefore $\sup _{h \in \R ^\star} \|u _h - u \|_{\lty{}} /|h| \leq C_P \|c \cdot \nabla _y v \|_{\lty{}}$ saying that $u \in \mathrm{D}(c \cdot \nabla _y)$ and $\|c \cdot \nabla _y u \|_{\lty{}} \leq C_P \|c \cdot \nabla _y v \|_{\lty{}}$. Notice also that we have
\[
b \cdot \nabla _y ( c \cdot \nabla _y u ) = c \cdot \nabla _y ( b \cdot \nabla _y u ) = c \cdot \nabla _y v
\]
with $\ave{c \cdot \nabla _y u } = c \cdot \nabla _y \ave{u} = 0$ saying that $(b \cdot \nabla _y ) ^{-1} (c \cdot \nabla _y v ) = c \cdot \nabla _y u$.
\end{proof}

\subsection{Regularity propagation under fast oscillating magnetic fields}
\label{RegPropagFastOsc}
We apply the previous general results to the operator \eqref{Equ17} acting on the phase space $(s,x,p) \in \R ^7$. We indicate a complete family of fields in involution with respect to $\partial _s + \frac{m \oc}{2} \teps \;^\perp x \cdot \nabla _p$. The reader can convince himself by direct computations.
\begin{pro}
\label{FamInv} The following fields are in involution with respect to $\partial _s + \frac{m \oc}{2} \teps \;^\perp x \cdot \nabla _p$
\[
c^1 \cdot \nabla _{(s,x,p)} = \partial _{x_1} - \frac{m \oc}{2} \tes \partial _{p_2},\;\;c^2 \cdot \nabla _{(s,x,p)} = \partial _{x_2} + \frac{m \oc}{2} \tes \partial _{p_1},\;\;c^3 \cdot \nabla _{(s,x,p)} = \partial _{x_3}
\]
\[
c^4 \cdot \nabla _{(s,x,p)} = \partial _{p_1},\;\;c^5 \cdot \nabla _{(s,x,p)} = \partial _{p_2},\;\;c^6 \cdot \nabla _{(s,x,p)} = \partial _{p_3}. 
\]
\end{pro}
\begin{pro}
\label{RegH1} Assume that $v \in \ker \ave{\cdot}$ such that $\nabla _x v, \nabla _p v \in (\ltpltxp{})^3$. Let us denote by $u$ the unique zero average solution of ${\cal T }u = v$. Then $\nabla _x u, \nabla _p u \in (\ltpltxp{}) ^3$ and
\[
\|\partial _{x_1} u\| \leq T \{ \|\partial _{x_1} v \| + m |\oc | \|\theta \|_{\linf{}} \|\partial _{p_2} v \| \}
\]
\[
\|\partial _{x_2} u\| \leq T \{ \|\partial _{x_2} v \| + m |\oc | \|\theta \|_{\linf{}} \|\partial _{p_1} v \| \}
\]
\[
\|\partial _{x_3} u\| \leq T  \|\partial _{x_3} v \|,\;\;\|\partial _{p_i}u \| \leq T \|\partial _{p_i} v \|,\;\;i \in \{1,2,3\}.
\]
\end{pro}
\begin{proof}
By the hypotheses we know that $v \in \cap _{i = 1} ^6 \mathrm{D}(c^i \cdot \nabla _{(s,x,p)})$ and therefore Proposition \ref{RegPropag} implies
\[
u \in \cap _{i = 1} ^6 \mathrm{D}(c^i \cdot \nabla _{(s,x,p)}),\;\;\|c ^i \cdot \nabla _{(s,x,p)} u \|\leq T \|c ^i \cdot \nabla _{(s,x,p)} v \|,\;\;i \in \{1,...,6\}.
\]
In particular we have $\nabla _x u, \nabla _p u \in (\ltpltxp{})^3$
\[
\|\partial _{x_3} u\| \leq T  \|\partial _{x_3} v \|,\;\;\|\partial _{p_i}u \| \leq T \|\partial _{p_i} v \|,\;\;i \in \{1,2,3\}
\]
\begin{eqnarray}
\|\partial _{x_1} u \| & = & \|c ^1 \cdot \nabla _{(s,x,p)} u + \frac{m \oc}{2} \theta \;\partial _{p_2} u \| \nonumber \\
& \leq & \|c ^1 \cdot \nabla _{(s,x,p)} u \| + \frac{m |\oc|}{2} \|\theta \|_{\linf{}} \|\partial _{p_2} u \| \nonumber \\
& \leq & T \|c ^1 \cdot \nabla _{(s,x,p)} v \| + T \frac{m |\oc|}{2} \|\theta \|_{\linf{}} \|\partial _{p_2} v \| \nonumber \\
& \leq & T \{ \|\partial _{x_1} v \| + m |\oc| \|\theta \|_{\linf{}} \|\partial _{p_2} v \|\} \nonumber 
\end{eqnarray}
and similarly $\|\partial _{x_2} u \| \leq T \{ \|\partial _{x_2} v \| + m |\oc| \|\theta \|_{\linf{}} \| \partial _{p_1} v \|\}$.
\end{proof}

\section{Limit model}
\label{LimMod}
\indent

We are ready now to investigate the limit model of \eqref{Equ4} as $\eps \searrow 0$. We appeal to the method introduced in \cite{BosTraSin} for general transport problems (see also \cite{BosAsyAna}) which combines Hilbert expansion and average properties. A rigorous convergence result is obtained using the notion of two-scale convergence. We analyze the properties of the limit model, in particular we establish the energy conservation and justify the confinement around the magnetic lines. 
The terms in the Hilbert expansion \eqref{Equ14} satisfy
\begin{equation}
\label{Equ22} {\cal T}\fz = 0
\end{equation}
\begin{equation}
\label{Equ23} \partial _t \fz + \frac{p}{m} \cdot \nabla _x \fz + ( eE (t,x) + \oc \tes{} \;^\perp p ) \cdot \nabla _p \fz + {\cal T}\fo = 0.
\end{equation}
The equation \eqref{Equ22} says that at any time $t \in \R_+$ the function $(s,x,p) \to \fz(t, s,x,p)$ belongs to $\ker {\cal T}$ and therefore
\[
\fz (t,s,x,p) = \gz \left ( t,x,p - \frac{m\oc}{2} \tes{} \;^\perp x\right ).
\]
The time evolution equation for $\fz$ is obtained from \eqref{Equ23} after eliminating $\fo$. Thanks to Proposition \ref{P2} we have for any $t \in \R_+$
\[
\partial _t \fz + \frac{p}{m} \cdot \nabla _x \fz + ( eE (t,x) + \oc \tes{} \;^\perp p ) \cdot \nabla _p \fz \in \ran {\cal T} = \ker \ave{\cdot}.
\]
Therefore \eqref{Equ23} is equivalent to
\begin{equation}
\label{Equ24} \ave{\partial _t \fz + \frac{p}{m} \cdot \nabla _x \fz + ( eE (t,x) + \oc \tes{} \;^\perp p ) \cdot \nabla _p \fz} = 0,\;\;t \in \R_+.
\end{equation}
It is easily seen that $\ave{\partial _t \fz } = \partial _t \ave{\fz} = \partial _t \fz$. It remains to compute the averages of the derivatives with respect to $x$ and $p$. For simplifying the computations we assume that $\fz$ is smooth but it can be shown that the limit model which we will obtain still holds true in the distribution sense. 
\begin{lemma}
\label{L1} Assume that $f(s,x,p) = g\left ( x,q = p - \frac{m\oc}{2} \tes{} \;^\perp x\right )$ is smooth. Then we have
\begin{eqnarray}
& & \ave{\frac{p}{m}\cdot \nabla _x f }  =  \left (\frac{q}{m} + \frac{\oc}{2} \ave{\theta} \;^\perp x    \right ) \cdot \nabla _x g + \left ( \frac{\oc}{2} \ave{\theta} q + \frac{m\oc ^2}{4} \ave{\theta ^2} \;^\perp x \right ) \cdot \;^\perp \nabla _q g \nonumber \\
& = & \left (\frac{p}{m} - \frac{\oc}{2} ( \theta - \ave{\theta}) \;^\perp x \right) \cdot \nabla _x f + \left (\frac{\oc}{2} ( \theta - \ave{\theta}) \;^\perp p + \frac{m \oc ^2 }{4} ( 2 \theta \ave{\theta} - \ave{\theta ^2} - \theta ^2) \;^{\perp \perp} x      \right ) \cdot \nabla _p f       \nonumber 
\end{eqnarray}
and
\begin{eqnarray}
 \ave{(eE(x) + \oc \theta \;^\perp p ) \cdot \nabla _p f } & = & eE \cdot \nabla _q g + \left ( \oc \ave{\theta} \;^\perp q  + \frac{m\oc ^2}{2} \ave{\theta ^2} \;^{\perp \perp } x \right ) \cdot \nabla _q g \nonumber \\
& = & \left (e E + \oc \ave{\theta} \;^\perp p + \frac{m \oc ^2}{2} ( \ave{\theta ^2} - \theta \ave{\theta}) \;^{\perp \perp} x     \right ) \cdot \nabla _p f. \nonumber  
\end{eqnarray}
\end{lemma}
\begin{proof}
We have, with the notation $^\perp \nabla _q = (\partial _{q_2}, - \partial _{q_1}, 0)$
\[
\nabla _x f = \nabla _x g + \frac{m\oc}{2} \tes{} \;^\perp \nabla _q g
\]
and therefore 
\[
\ave{\frac{p}{m}\cdot \nabla _x f } = \ave{p} \cdot \frac{\nabla _x g}{m} + \frac{\oc}{2} \ave{p \theta} \cdot \;^\perp \nabla _q g
\]
since the derivatives of $g$ are constant along the characteristic flow \eqref{Equ17Bis}. By the definition of the average operator we have
\begin{eqnarray}
\ave{p} & = & \frac{1}{T} \int _0 ^T \left \{p - \frac{m \oc}{2} \tes{} \;^\perp x + \frac{m\oc}{2}\theta (\tau) \;^\perp x   \right \}\;\mathrm{d}\tau \nonumber \\
& = & q + \frac{m\oc}{2} \ave{\theta} \;^\perp x \nonumber
\end{eqnarray}
where $q = p - \displaystyle\frac{m\oc}{2} \tes{} \;^\perp x$. Similarly we obtain
\begin{eqnarray}
\ave{p\theta} & = & \frac{1}{T} \int _0 ^T \left \{ p - \frac{m \oc}{2} \tes{} \;^\perp x + \frac{m\oc}{2}\theta (\tau) \;^\perp x    \right \}\theta (\tau) \;\mathrm{d}\tau \nonumber \\
& = & q \ave{\theta} + \frac{m\oc}{2} \ave{\theta ^2} \;^\perp x \nonumber 
\end{eqnarray}
and finally 
\[
\ave{\frac{p}{m}\cdot \nabla _x f } = \left (\frac{q}{m} + \frac{\oc}{2} \ave{\theta} \;^\perp x    \right ) \cdot \nabla _x g + \left ( \frac{\oc}{2} \ave{\theta} q + \frac{m\oc ^2}{4} \ave{\theta ^2} \;^\perp x \right ) \cdot \;^\perp \nabla _q g.
\]
The second statement follows easily observing that $\nabla _p f = \nabla _q g$ and therefore 
\begin{eqnarray}
\ave{(eE(x) + \oc \theta \;^\perp p ) \cdot \nabla _p f } & = & eE(x) \cdot \nabla _q g + \oc \ave{\theta \;^\perp p} \cdot \nabla _q g \nonumber \\
& = &  eE(x) \cdot \nabla _q g + \left ( \oc \ave{\theta} \;^\perp q  + \frac{m\oc ^2}{2} \ave{\theta ^2} \;^{\perp \perp } x \right ) \cdot \nabla _q g.\nonumber
\end{eqnarray}
\end{proof}
Combining the previous computations yields the transport equation
\begin{equation}
\label{Equ25} \partial _t \gz + \left (\frac{q}{m} + \frac{\oc}{2} \ave{\theta} \;^\perp x     \right )\cdot \nabla _x \gz + \left ( e E + \frac{\oc}{2} \ave{\theta} \;^\perp q + \frac{m\oc ^2}{4} \ave{\theta ^2} \;^{\perp \perp } x   \right ) \cdot \nabla _q \gz= 0
\end{equation}
which is exactly \eqref{Equ10}. Performing the change of unknown $\fz (t,s,x,p) = \gz (t,x, p - \frac{m\oc}{2} \tes{} \;^\perp x)$ leads to
\begin{eqnarray}
\label{Equ26} & & \partial _t \fz + \left (\frac{p}{m} - \frac{\oc}{2} ( \tes{} - \ave{\theta})\;^\perp x    \right ) \cdot \nabla _x \fz   \\
& + & \left (eE(t,x) + \frac{\oc}{2} ( \tes{} + \ave{\theta}) \;^\perp p + \frac{m\oc ^2}{4} ( \ave{\theta ^2} - \theta ^2 (s)) \;^{\perp \perp} x    \right ) \cdot \nabla _p \fz = 0.\nonumber
\end{eqnarray}
Based on the concept of two-scale convergence, introduced in \cite{Guetseng} and developed in \cite{Allaire92} we prove a two-scale convergence result of the solutions in \eqref{Equ4} towards \eqref{Equ26}, supplemented by an appropriate initial condition.
\begin{defi}
\label{TwoScaleConv} Let $(f^n (t,x,p))_{n\in \N ^\star}$ be a sequence in $L^2([0,I];\ltxp{})$. We say that $(f^n)_n$ two-scale converges towards some function $f^0 \in L^2([0,I]; \ltpltxp{})$ iff we have
\[
\limn \int _0 ^I \!\!\int_{\R^3}\!\!\int _{\R^3} \!\!f ^n (t,x,p) \varphi (t, \s, x, p)\;\mathrm{d}p\mathrm{d}x\mathrm{d}t = \int _0 ^I \frac{1}{T} \intsxp{(f^0 \varphi )(t,s,x,p)}\mathrm{d}t
\]
for any test function $\varphi \in L^2([0,I];C^0 _{\#}(\R_s; \ltxp{}))$ (here $C^0 _{\#} (\R_s; X)$ stands for the set of continuous $T$ periodic functions with values in the normed linear space $X$).
\end{defi}
Adapting the arguments in \cite{Allaire92} we obtain
\begin{pro}
\label{CompactTwoScales} Let $(\fe (t,x,p))_{\eps >0}$ be a bounded family in $L^2([0,I];\ltxp{})$. Then there is a sequence $\eps _n \searrow 0$ and a function $f^0 \in L^2([0,I]; \ltpltxp{})$ such that $(f ^{\eps _n})_n$ two-scale converges towards $f^0$.
\end{pro}
\begin{proof} (of Theorem \ref{MainResultWeak}) 
We use the weak formulation of \eqref{Equ4} with the test function $\eta (t) \varphi(\s, x, p)$ where $\eta \in C^1 _c (\R_+)$ and $\varphi \in C^1 _c (\R/T\Z \times \R ^3 \times \R^3)$. We obtain
\begin{eqnarray}
& & - \eta (0) \intxp{\fin (x,p) \varphi (0,x,p)} - \int _{\R_+}\eta ^{\;\prime} (t) \intxp{\fe (t,x,p) \varphi (\s, x, p)}\mathrm{d}t \nonumber \\
& & - \int _{\R_+} \!\!\!\eta  \intxp{\fe (t,x,p) \left [ \frac{1}{\eps} \partial _s \varphi (\s, x, p) + \frac{p}{m} \cdot \nabla _x \varphi (\s,x,p)\right ] }\mathrm{d}t \nonumber \\
& & - \int _{\R_+} \!\!\!\eta  \intxp{\fe (t,x,p) \left [eE + \frac{m\oc}{2\eps} \tept{} \;^\perp x + \oc \tet{}\;^\perp p \right] \cdot \nabla _p \varphi (\s, x, p)  }\mathrm{d}t \nonumber \\
& & = 0.\nonumber
\end{eqnarray}
Multiplying the previous formulation by $\eps$, one gets by two-scale convergence (after extraction eventually) 
\[
\int _{\R_+} \eta (t) \frac{1}{T} \intsxp{f^0 (t,s,x,p) {\cal T}\varphi (s,x,p)}\mathrm{d}t = 0
\]
saying that $f^0(t,\cdot, \cdot, \cdot) \in \ker {\cal T}$ for any $t \in \R_+$. We use now the same weak formulation, but with $\varphi \in C^1 _c (\R/T\Z \times \R ^3 \times \R^3) \cap \ker {\cal T}$. For example take $\varphi (s,x,p) = \psi (x, p - \frac{m \oc }{2} \tes{} \;^\perp x)$ with $\psi \in C^1 _c ( \R ^3 \times \R ^3)$, let us say $\supp \psi \subset \{ (x,q) : \max \{ |x|, |q|\} \leq R\}$. In this case $\supp \varphi \subset \{ (s,x,p) : \max \{|x|, |p|\} \leq ( 1 + m |\oc |\|\theta \|_{\linf{}} /2 ) R \}$. Since $\varphi (s,x,p) \in \ker {\cal T}$ we have
\[
\frac{1}{\eps } \partial _s \varphi (\s, x, p) + \frac{m\oc}{2\eps} \tept{} \;^\perp x \cdot \nabla _p \varphi (\s, x, p) = 0
\]
and the weak formulation becomes
\begin{eqnarray}
\label{Equ27} && - \eta (0) \intxp{\fin (x,p) \varphi (0,x,p)} - \int _{\R_+}\eta ^{\;\prime} (t) \intxp{\fe (t,x,p) \varphi (\s, x, p)}\mathrm{d}t \nonumber \\
& & - \int _{\R_+} \eta (t) \intxp{\fe  \left [ \frac{p}{m} \cdot \nabla _x \varphi (\s,x,p) + \left ( eE + \oc \tet{}\;^\perp p \right) \cdot \nabla _p \varphi (\s, x, p)  \right ]}\mathrm{d}t \nonumber \\
& & = 0.
\end{eqnarray}
As before, by two-scale convergence one gets
\begin{eqnarray}
& & \lime{} \int _{\R_+}\eta ^{\;\prime} (t) \intxp{\fe (t,x,p) \varphi (\s, x, p)}\mathrm{d}t \nonumber \\
&  & =  \int _{\R_+}\eta ^{\;\prime} (t) \frac{1}{T}\intsxp{\fz(t,s,x,p) \varphi (s,x,p) }\mathrm{d}t. \nonumber
\end{eqnarray}
Similarly one gets
\begin{eqnarray}
& & \lime \int _{\R_+} \!\!\! \eta  \intxp{\fe \left [ \frac{p}{m} \cdot \nabla _x \varphi (\s,x,p) + \left ( eE + \oc \tet{}\;^\perp p \right) \cdot \nabla _p \varphi (\s, x, p)  \right ]} \mathrm{d}t \nonumber \\
& & = \int _{\R_+} \!\frac{\eta}{T} \intsxp{\fz (t,s,x,p) \left [ \frac{p}{m} \cdot \nabla _x \varphi  + \left ( eE + \oc \tes{}\;^\perp p \right) \cdot \nabla _p \varphi  \right ] }\mathrm{d}t. \nonumber
\end{eqnarray}
Therefore  \eqref{Equ27} implies
\begin{eqnarray}
\label{EquFormWeak}
& & - \eta (0) \intxp{\fin (x,p) \varphi (0,x,p) }- \int_{\R_+} \frac{\eta ^{\;\prime} (t)}{T} \intsxp{\fz (t,s,x,p) \varphi (s,x,p)}\mathrm{d}t \nonumber \\
& = & \int _{\R _+} \frac{\eta (t)}{T} \intsxp{\fz (t,s,x,p)\left [ \frac{p}{m} \cdot \nabla _x \varphi  + \left ( eE + \oc \tes{}\;^\perp p \right) \cdot \nabla _p \varphi   \right ] }\mathrm{d}t.
\end{eqnarray}
We transform the term involving the initial condition. Since $\varphi \in \ker {\cal T}$ we have for any $s \in \R$
\[
\varphi (0,x,p) = \varphi \left (s, x, p - \frac{m\oc}{2}\theta (0) \;^\perp x +   \frac{m\oc}{2}\tes{} \;^ \perp x   \right )
\]
implying that 
\begin{eqnarray}
\label{EquICTerm} T \!\!\intxp{\fin  \varphi (0,x,p)}
& \!\!\!\!=\!\!\!\! & \!\!\!\intsxp{\fin (x,p) \varphi \left  (s, x, p - \frac{m\oc}{2}( \theta (0) - \tes{}) \;^ \perp x  \right )} \nonumber \\
& \!\!\!\!=\!\!\!\! & \!\!\!\intsxp{\fin \left (x, p - \frac{m\oc}{2} (\tes{} - \theta (0))\;^ \perp x     \right )\varphi (s,x,p)}.\nonumber \\
&& 
\end{eqnarray}
We transform now the right hand side of \eqref{EquFormWeak} thanks to Lemma \ref{L1}. Indeed, since $f(t, \cdot, \cdot, \cdot) \in \ker {\cal T}$, for any $t \in \R_+$, we have by the variational formulation of the average operator in Proposition \ref{P1}
\begin{eqnarray}
\label{EquRHS} & & \intsxp{\fz (t)\left [ \frac{p}{m} \cdot \nabla _x \varphi  + \left ( eE + \oc \tes{}\;^\perp p \right) \cdot \nabla _p \varphi   \right ]} \nonumber \\
& =  & \intsxp{\fz \ave{\frac{p}{m} \cdot \nabla _x \varphi  + \left ( eE + \oc \tes{}\;^\perp p \right) \cdot \nabla _p \varphi  }}\nonumber \\
& = &   \intsxp{\fz (t)  \left [\left (\frac{p}{m} - \frac{\oc}{2} ( \tes{} - \ave{\theta})\;^\perp x    \right ) \cdot \nabla _x \varphi  \right.  \nonumber \\
& + &    \left. \left (eE + \frac{\oc}{2} ( \tes{} + \ave{\theta}) \;^\perp p + \frac{m\oc ^2}{4} ( \ave{\theta ^2} - \theta ^2 (s)) \;^{\perp \perp} x    \right ) \cdot \nabla _p \varphi     \right ]}.
\end{eqnarray}
Combining now \eqref{EquFormWeak}, \eqref{EquICTerm}, \eqref{EquRHS} implies that $\fz$ is a weak solution of the transport problem \eqref{EquNew3}, \eqref{EquNew4}. Equivalently, the function $\gz$ such that $\fz (t,s,x,p) = \gz (t,x,p - m \oc /2 \;\tes{} \;^\perp x )$ is a weak solution of the transport equation \eqref{Equ25}, supplemented  by the initial condition 
\begin{equation}
\label{Equ31} \gz ( 0,x,q) = \fin \left ( x, q + \frac{m\oc}{2} \theta (0) \;^\perp x \right ).
\end{equation}
\end{proof}
Notice that the model \eqref{Equ25}, \eqref{Equ31} is posed in a six dimensional phase space whereas the model \eqref{EquNew3}, \eqref{EquNew4} acts on a seven dimensional phase space. Thus, at least for the numerical point of view it is preferable to appeal to \eqref{Equ25}, \eqref{Equ31}. 

It is interesting to observe that, as $\eps \searrow 0$, the kinetic energy of $\fz(t, t /\eps, \cdot, \cdot) \approx \fe (t, \cdot, \cdot)$ can be decomposed into kinetic energy and elastic energy associated to the density $\gz$
\[
W^0 (t) = \intxq{\left \{ \frac{1}{2m} \left | q + \frac{m \oc}{2} \ave{\theta} \;^\perp x    \right |^ 2 + \frac{m \oc ^2}{4} ( \ave{\theta ^2} - \ave{\theta} ^2 ) \frac{|^\perp x|^2}{2} \right \}\gz(t,x,q)}.
\]
Moreover, when the electric potential solves the Poisson equation corresponding to the concentration
\[
\intp{\fz(t, t/\eps, x, p)} = \intp{\gz\left (t, x, p - \frac{m \oc}{2}\tet{} \;^\perp x   \right )} = \intq{\gz(t,x,q)}
\]
the total energy is preserved. In order to simplify our computations we work with smooth solutions.
\begin{pro}
\label{ComSuppLoc}
Assume that $\fin$ is nonnegative, $\fin \in W^{1, \infty} (\R ^3 \times \R ^3)$ such that 
\[
\intxp{( 1 + |^ \perp x | ^2 + |p|^2) \fin (x,p)} < +\infty .
\]
i) If $E \in L^1 _{\mathrm{loc}} ( \R_+; W^{1, \infty} (\R^3))$ is a given electric field then
\begin{equation}
\intxp{\frac{|p|^2}{2m}\fz(t,\s, x, p) } \rightharpoonup W^0 (t)\;\mbox{ weakly }\;\star \;\mbox{ in }\;  L^\infty _{\mathrm{loc}} ( \R_+)\;\mbox{ as }\; \eps \searrow 0.
\end{equation}
ii) If $\fin $ has compact support and $E \in L^1 _{\mathrm{loc}} ( \R_+; W^{1, \infty} (\R^3))$, then there is a continuous nondecreasing function $R$ depending on $t$, such that
\[
\supp \gz (t, \cdot, \cdot) \subset \{(x,q) \;:\; |x| \leq R(t),\;\;|q| \leq R(t)   \},\;\;t \in \R_+.
\]
iii) If $E$ belongs to $ \in L^1 _{\mathrm{loc}} ( \R_+; W^{1, \infty} (\R^3))$ and  solves the Poisson equation
\begin{equation}
\label{Equ33} E = - \nabla _x \phi,\;\;
- \Delta _x \phi (t) = \frac{e}{\eps _0} \left ( \int _{\R ^3} \gz (t,x,q) \;\mathrm{d}q - n _0 (x) \right )
\end{equation}
such that $E(0,\cdot) \in (L^2(\R^3))^3$, 
then 
\[
\frac{d}{dt} \left \{ W^0  + \frac{\eps _0}{2} \intx{|E|^2}  \right \} = 0.
\]
\end{pro}
\begin{proof}
i) It is easily seen that 
\begin{eqnarray}
\intxp{\fz(t,\s, x, p) } & = &  \intxp{\gz\left (t, x, p - \frac{m\oc}{2} \tet{} \;^\perp x    \right )} \nonumber \\
& = &  \intxq{\gz (t,x,q)} \nonumber \\
& = & \intxq{\gz (0,x,q)} \nonumber \\
& = &  \intxq{\fin \left (x, q + \frac{m\oc}{2}\theta (0) \;^\perp x    \right )} \nonumber \\
& = & \intxp{\fin (x,p)} \nonumber 
\end{eqnarray}
saying that $\fz(t, t/\eps, \cdot, \cdot)$ belongs to $L^1 (\R _x ^3 \times \R _p ^3)$ uniformly in $t\in \R_+,\; \eps >0$. 
We consider the function 
\begin{eqnarray}
\chi (x,q) & = &  \frac{1}{T} \int _0 ^T \frac{1}{2m} \left | q + \frac{m \oc }{2} \theta (s) \;^\perp x \right | ^2 \;\mathrm{d}s \nonumber \\
& = & 
\frac{|q|^2}{2m} + \frac{m \oc ^2}{4} \ave{\theta ^2} \frac{|^\perp x|^2}{2} + \frac{\oc}{2} \ave{\theta} ( q\cdot \;^\perp x) \nonumber \\
& = & \frac{1}{2m} \left | q + \frac{m \oc}{2} \ave{\theta} \;^\perp x  \right | ^2 + \frac{m \oc ^2}{4} ( \ave{\theta ^2} - \ave{\theta} ^2 ) \frac{|^\perp x|^2}{2}.\nonumber 
\end{eqnarray}
Observe that 
\begin{equation}
\label{EquMasterInv}
\left (\frac{q}{m} + \frac{\oc}{2} \ave{\theta} \;^\perp x    \right )\cdot \nabla _x \chi  + \left ( eE + \frac{\oc}{2} \ave{\theta} \;^\perp q  + \frac{m \oc ^2}{4} \ave{\theta ^2} \;^{\perp \perp }x \right ) \cdot \nabla _q \chi = eE \cdot \left (\frac{q}{m} + \frac{\oc}{2} \ave{\theta} \;^\perp x    \right )
\end{equation}
and thus multiplying \eqref{Equ25} by $\chi$ we deduce 
\begin{eqnarray}
\label{Equ32} \frac{d}{dt} \intxq{\gz \chi (x,q)} & = & \intxq{\gz(t,x,q) eE \cdot \left (\frac{q}{m} + \frac{\oc}{2} \ave{\theta} \;^\perp x    \right )}  \\
& \leq & \intxq{\gz |eE| \left \{ \frac{1}{2m} \left | q + \frac{m \oc}{2} \ave{\theta} \;^\perp x  \right | ^2 + \frac{1}{2m} \right \}}\nonumber \\
& \leq & \|e E (t)\|_{\linf{}} \intxq{\gz \chi } + \frac{\|e E(t) \|_{\linf{}}}{2m}  \intxp{\fin}.\nonumber
\end{eqnarray}
By Gronwall's lemma it follows that $\intxq{\;\gz(\cdot, x, q) \chi (x,q)} \in L^\infty _{\mathrm{loc}}( \R_+)$ 
\begin{eqnarray}
\intxq{\gz (t,x,q) \chi } & \leq & \left ( \intxq{\gz (0,x,q) \chi  } + \frac{\| e E\|_{L^1([0,t];\linf{})}}{2m} \intxp{\fin }   \right ) \nonumber \\
& \times & \exp \left ( \int _0 ^t \|e E(s)\|_{\linf{}} \;\mathrm{d}s    \right ). \nonumber
\end{eqnarray}
For any function $\eta \in L^1 _{\mathrm{loc}} ( \R_+)$ we can write
\begin{eqnarray}
& & \lime \int _{\R_+} \eta (t) \intxp{\fz (t, \s, x, p)\frac{|p|^2}{2m}}\mathrm{d}t  \nonumber \\
&  = & \lime \int _{\R_+} \eta (t) \intxp{\gz \left (t, x, p- \frac{m \oc}{2} \tet{} \;^\perp x \right )\frac{|p|^2}{2m}}\mathrm{d}t\nonumber \\
& = & \lime{}  \int _{\R_+} \eta (t) \intxq{\gz  (t, x, q)\frac{1}{2m}\left |q + \frac{m \oc }{2}\tet{} \;^\perp x   \right | ^2}\mathrm{d}t\nonumber \\
& = & \int _{\R_+} \eta (t) \intxq{\gz  (t, x, q) \chi (x,q) } \mathrm{d}t
\nonumber 
\end{eqnarray}
saying that 
\[
\intxp{\frac{|p|^2}{2m}\fz(t,\s, x, p) } \rightharpoonup W^0 (t)\;\mbox{ weakly }\; \star \;\mbox{ in }\; L^\infty _{\mathrm{loc}} ( \R_+) \;\mbox{ as }\; \eps \searrow 0.
\]
ii) Assume that  $\supp \fin \subset \{(x,p) \;:\; |x| \leq \Rin, |p| \leq \Rin   \}$ with $\Rin >0$. Therefore
\begin{eqnarray}
\supp \gz(0, \cdot, \cdot) & \subset &  \{ (x,q )\;:\;|x| \leq \Rin,\;\;|q| \leq \left ( 1 + m |\oc|/2\; \|\theta \|_{\linf{}} \right ) \Rin  \} \nonumber \\
&  \subset  & \{ (x,q)\;:\; \sqrt{\chi (x,q)} \leq \nu \Rin, |x_3| \leq \nu \Rin\}
\end{eqnarray}
for some $\nu >0$. We consider a function $\xi$ satisfying 
\[
\xi \in C^1 (\R),\;\;0 \leq \xi \leq 1,\;\;\supp \xi = \R_+,\;\;\xi ^{\;\prime} \geq 0.
\]
Applying the weak formulation of \eqref{Equ25} with the test function $(t,x,q) \to \xi ( \sqrt{\chi (x,q)} - \alpha (t))$, with $\alpha \in C^1 (\R_+)$ yields by \eqref{EquMasterInv}
\begin{eqnarray}
& & \intxq{\gz (t,x,q) \xi ( \sqrt{\chi (x,q)} - \alpha (t) )} - \intxq{\gz (0,x,q) \xi ( \sqrt{\chi (x,q)} - \alpha (0)) }\nonumber \\
& = & \int _0 ^t \intxq{\gz (\tau, x, q) \xi ^{\;\prime} \left \{- \alpha ^{\;\prime} (\tau) + \frac{eE(\tau, x)}{2 \sqrt{\chi (x,q)}} \cdot \left ( \frac{q}{m} + \frac{\oc }{2} \ave{\theta} \;^\perp x \right )   \right \}}\mathrm{d}\tau \nonumber \\
& \leq & \int _0 ^t \intxq{\gz (\tau, x, q) \xi ^{\;\prime} \left \{- \alpha ^{\;\prime} (\tau) + \frac{\|eE (\tau)\|_{\linf{}}}{\sqrt{2m}} \right \}}\mathrm{d}\tau. \nonumber 
\end{eqnarray}
We take $\alpha (0) = \nu \Rin$ and $\alpha ^{\;\prime} (\tau) = \frac{\|eE (\tau)\|_{\linf{}}}{\sqrt{2m}}$, $\tau \in \R_+$. Since $\gz (0,x,q) \xi  ( \sqrt{\chi (x,q)} - \alpha (0) ) = 0$ one gets
\[
\intxq{\gz (t,x,q) \xi  ( \sqrt{\chi (x,q)} - \alpha (t) )} \leq 0
\]
implying that $\supp \gz (t, \cdot, \cdot)  \subset \{(x,q):\sqrt{\chi (x,q)} \leq \alpha (t)\}$. Similarly, applying the weak formulation of \eqref{Equ25} with the test function $(t,x,q) \to \xi ( |x_3 - t q_3 /m| - \beta (t))$, $\beta \in C^1 (\R_+)$ we obtain
\begin{eqnarray}
& & \intxq{\gz (t,x,q) \xi (|x_3 - t q_3 /m|  - \beta (t) )} - \intxq{\gz (0,x,q) \xi ( |x_3| - \beta (0)) }\nonumber \\
& = & \int _0 ^t \intxq{\gz (\tau, x, q) \xi ^{\;\prime} \left \{- \beta ^{\;\prime} (\tau) - \frac{\tau e E_3(\tau,x ) }{m} \mathrm{sgn} (x_3 - t q_3 /m) \right \}}\mathrm{d}\tau \nonumber \\
& \leq & \int _0 ^t \intxq{\gz (\tau, x, q) \xi ^{\;\prime} \left \{- \beta ^{\;\prime} (\tau) + \frac{\tau \|eE_3 (\tau)\|_{\linf{}}}{m} \right \}}\mathrm{d}\tau. \nonumber 
\end{eqnarray}
Taking $\beta (0) = \nu \Rin$ and $\beta ^{\;\prime} (\tau) = \frac{\tau \|eE_3 (\tau)\|_{\linf{}}}{m}$, $\tau \in \R_+$ we deduce that
\[
\supp \gz (t, \cdot, \cdot)  \subset \{(x,q)\;:\; |x_3 - t q_3 /m| \leq \beta (t)\}.
\]
We have proved that 
\begin{eqnarray}
\supp \gz (t)  & \subset &   \{ (x,q) \;:\; \sqrt{\chi (x,q)} \leq 
\nu \Rin + \frac{|e|}{\sqrt{2m}} \|E\|_{L^1([0,t];\linf{}) }\} \nonumber \\
& \cap & \{ (x,q) \;:\;  
|x_3 | \leq \nu \Rin + t \frac{|q_3|}{m} + \frac{t \|e E _3\|_{L^1([0,t];\linf{})}}{m}\} \nonumber \\
& \subset & \{ (x,q) \;:\;  |x| \leq R(t),\;\;|q| \leq R(t) \} \nonumber 
\end{eqnarray}
for some continuous nondecreasing function $R(t)$. Notice also that for any $t \in \R_+$
\[
\supp \fz (t, \cdot, \cdot, \cdot) \subset \{ (s,x,p) \;:\; |x| \leq R(t),\;\;|p| \leq ( 1 + m |\oc| /2 \;\|\theta \|_{\linf{}} ) R(t) \}.
\]
iii) Assume now that the electric potential solves the Poisson equation \eqref{Equ33}. By standard computations involving the continuity equation
\[
\partial _t \intq{\gz} + \Div \intq{\left ( \frac{q}{m} + \frac{\oc}{2} \ave{\theta} \;^\perp x   \right ) \gz } = 0
\]
one gets
\begin{equation}
\label{Equ34} \frac{d}{dt} \frac{\eps _0 }{2} \intx{|E|^2} = - \intx{eE (t,x) \cdot \intq{\left ( \frac{q}{m} + \frac{\oc}{2} \ave{\theta} \;^\perp x   \right ) \gz}}.
\end{equation}
Combining \eqref{Equ32}, \eqref{Equ34} yields
\[
\frac{d}{dt} \left \{\intxq{\gz (t,x,q) \chi (x,q) } + \frac{\eps _0 }{2} \intx{|E(t,x) |^2}   \right \} = 0.
\]
\end{proof}
In the sequel we inquire about the confinement properties of the limit model \eqref{Equ25}, \eqref{Equ31} (resp. \eqref{EquNew3}, \eqref{EquNew4}). We exploit here the invariants of \eqref{Equ25}. Let us consider for the moment that the potential is stationary. In this case notice that \eqref{Equ25} writes under the Hamiltonian form
\[
\partial _t \gz + \nabla _q H  \cdot \nabla _x \gz - \nabla _x H \cdot \nabla _q \gz = 0
\]
with the Hamiltonian 
\[
H(x,q) = \chi (x, q) + e \phi (x) = \frac{1}{2m} \left   |q + \frac{m \oc }{2} \ave{\theta} \;^\perp x   \right |^2 + \frac{m \oc ^2}{4} ( \ave{\theta ^2} - \ave{\theta} ^2 ) \frac{|^\perp x |^2}{2} + e \phi (x).
\]
In particular $H$ is a stationary solution of \eqref{Equ25} or equivalently $H$ is an invariant of the characteristic flow of \eqref{Equ25}
\[
\frac{dX}{dt} = \nabla _q H (X(t), Q(t)),\;\;\frac{dQ}{dt} = -\nabla _x H (X(t), Q(t)).
\]
Under additional hypotheses on the electric potential $\phi$ it is easily seen that the plasma remains confined in a bounded region around the magnetic field lines. Indeed, assume that the hypothesis in Theorem \ref{MainResult2} holds true
\begin{equation}
\label{Equ35} \lim _{|^\perp x| \to +\infty} \left \{e \phi (x)  + \frac{m\oc ^2}{4} ( \ave{\theta ^2} - \ave{\theta} ^2 ) \frac{|^\perp x|^2}{2}    \right \} = + \infty
\end{equation}
uniformly with respect to $x_3$ and that at the initial time we have
\begin{equation}
\label{Equ36} 0 \leq \gz (0, x, q) \leq \lambda ( H(x,q)),\;\;(x,q) \in \R ^3 \times \R ^3
\end{equation}
for some nonnegative profile $\lambda \in C^1  (\R)$, vanishing on $[L,+\infty[$. By the maximum principle we deduce that
\[
0 \leq \gz(t,x,q) \leq \lambda ( H(x,q)),\;\;(t,x,q) \in \R_+ \times \R ^3 \times \R ^3
\]
which guarantees the compactness of the support of $\gz (t)$ along the orthogonal directions to the magnetic lines, uniformly in time. Indeed, we have
\[
\supp \gz (t) \subset \{ (x,q) \;:\; H(x,q) \leq L\} \subset \{(x,q) \;:\; e\phi (x) + \frac{m\oc^2}{4} ( \ave{\theta ^2} - \ave{\theta} ^2 ) \frac{|^\perp x|^2}{2} \leq L \}.
\]
By the hypothesis \eqref{Equ35} there is $R >0$ such that for any $x \in \R^3$ verifying  $|^\perp x | > R$ we have
\[
e\phi (x) + \frac{m\oc^2}{4} ( \ave{\theta ^2} - \ave{\theta} ^2 ) \frac{|^\perp x|^2}{2} > L
\]
and finally 
\[
\supp \gz (t) \subset \{(x,q) \;:\; |^\perp x | \leq R \},\;\;\forall\;t \in \R_+.
\]

\begin{remark}
If the electric potential depends only on $|^\perp x|$ and $x_3$ then $(q \cdot \;^\perp x)$ is another invariant of \eqref{Equ25}. If the electric potential depends only on $|^\perp x|$ then $q_3$ is an invariant of \eqref{Equ25} as well. 
\end{remark}
When the electric potential depends on time, the previous Hamiltonian becomes $H(t,x,q) = \chi (x,q) + e \phi (t,x)$ and therefore we obtain
\begin{equation}
\label{EquTimDepH} \partial _t H + \left (\frac{q}{m} + \frac{\oc}{2} \ave{\theta} \;^\perp x     \right )\cdot \nabla _x H + \left ( e E + \frac{\oc}{2} \ave{\theta} \;^\perp q + \frac{m\oc ^2}{4} \ave{\theta ^2} \;^{\perp \perp } x   \right ) \cdot \nabla _q H = e \partial _t \phi.
\end{equation}
In this case we need to construct a particular super-solution for \eqref{Equ25}.

\begin{proof} (of Theorem \ref{MainResult2} for time dependent electric potential) We work in the phase space $(x,q)$, using the equation \eqref{Equ25} and the initial condition \eqref{Equ31}. By the hypotheses we know that
\[
\gz (0,x, q) \leq \lambda ( \chi (x,q) + e \phi (0,x)),\;\;(x,q) \in \R ^3 \times \R ^3.
\]
Let us consider the function $h$ defined for any $(t,x,q) \in \R_+ \times \R ^3 \times \R^3$ by
\[
h(t,x,q) = \lambda \left ( \chi (x,q) + e \phi (t,x) - \int _0 ^t \sup _{y \in \R^3} \{ e \partial _t \phi (s,y) \}\;\mathrm{d}s \right ).
\]
Observe that 
\[
h(0,x,q) = \lambda ( \chi (x,q) + e \phi (0,x)) \geq \gz (0, x, q),\;\;(x,q) \in \R ^3 \times \R ^3.
\]
Taking into account \eqref{EquTimDepH} it is easily seen, by the monotonicity of $\lambda$, that
\[
\partial _t h + \left (\frac{q}{m} + \frac{\oc}{2} \ave{\theta} \;^\perp x     \right )\cdot \nabla _x h + \left ( e E + \frac{\oc}{2} \ave{\theta} \;^\perp q + \frac{m\oc ^2}{4} \ave{\theta ^2} \;^{\perp \perp } x   \right ) \cdot \nabla _q h \geq 0.
\]
By the maximum principle we deduce that 
\[
\gz (t,x,q) \leq h(t,x,q),\;\;(t,x,q) \in \R_+ \times \R ^3 \times \R^3.
\]
By the hypothesis \eqref{EquNew5} there is $R >0$ such that for any $(t,x) \in \R_+ \times \R ^3$, with $|^\perp x| > R$ we have
\[
e \phi (t, x) - \int _0 ^t \sup _{y \in \R^3} \{ e \partial _t \phi (s,y) \}\;\mathrm{d}s + \frac{m\oc ^2}{4} ( \ave{\theta ^2} - \ave{\theta} ^2 ) \frac{|^\perp x|^2}{2} > L
\]
implying that for any $t \in \R_+$
\begin{eqnarray}
\supp \gz (t) & \subset & \{ (x,q) \;:\; \chi (x,q) + e \phi (t, x) - \int _0 ^t \sup _{y \in \R^3} \{ e \partial _t \phi (s,y) \}\;\mathrm{d}s \leq L\} \nonumber \\
& \subset & \{ (x,q) \;:\;e \phi (t, x) - \int _0 ^t \sup _{y \in \R^3} \{ e \partial _t \phi (s,y) \}\;\mathrm{d}s + \frac{m\oc ^2}{4} ( \ave{\theta ^2} - \ave{\theta} ^2 ) \frac{|^\perp x|^2}{2} \leq L\} \nonumber \\
& \subset &  \{ (x,q) \;:\;|^ \perp x| \leq R \}.\nonumber
\end{eqnarray}
Since $\fz (t,s,x,p) = \gz (t, x, p - m \oc \tes /2 \;^\perp x)$ finally one gets
\[
\supp \fz (t, s, \cdot, \cdot) \subset \{ (x,p) \;:\; |^ \perp x | \leq R\},\;\;t \in \R_+, s \in \R.
\]
In particular 
\[
\supp \fz (t, \s, \cdot, \cdot) \subset \{ (x,p) \;:\; |^ \perp x | \leq R\},\;\;t \in \R_+,\;\; \eps >0.
\]
\end{proof}

\section{Asymptotic behaviour}
\label{AsyBeh}
\indent

The aim of this section is to justify rigorously the Hilbert expansion \eqref{Equ14}. More precisely we intend to prove that $\fe (t,x,p) =  \fz (t,\s, x, p) + {\cal O}(\eps)$ strongly in $\ltxp{}$, uniformly for $t$ in bounded intervals. The idea is to introduce the solution $\Fe = \Fe (t,s,x,p)$ of the transport problem
\begin{equation}
\label{Equ40} \partial _t \Fe + \frac{p}{m} \cdot \nabla _x \Fe + (eE(t,x) + \oc \tes{} \;^\perp p ) \cdot \nabla _p \Fe + \frac{1}{\eps} {\cal T} \Fe = 0
\end{equation}
with the initial condition
\begin{equation}
\label{Equ41} \Fe (0,s,x,p) = \fin \left ( x, p - \frac{m\oc}{2}\tes{} \;^\perp x + \frac{m\oc}{2}\theta(0) \;^\perp x\right )
\end{equation}
and to observe that $\Fe(t,\s, x, p)$ satisfies \eqref{Equ4}, \eqref{EquIC}, saying that $\Fe (t,\s, x, p) = \fe (t,x,p)$. We start by estimating the error between $\Fe (t,s,x,p) $ and $\fz (t,s,x,p)$. We also prove that for any $I \in \R_+$ the functions $\{ F ^\eps (t): t \in [0,I], \eps >0\}$ are uniformly compactly supported. 
\begin{pro}
\label{P4} Assume that $E \in L^1 _{\mathrm{loc}} ( \R_+; W^{2,\infty}(\R^3))$, $\partial _t E \in L^1_{\mathrm{loc}} (\R_+; \linf{})$ for any $I \in \R_+$, $\fin \in W^{2,\infty} ( \R ^3 \times \R ^3)$ and $\supp \fin $ is compact. Then for any $I \in \R_+$, there is a constant $C_1(I)$ such that 
\[
\|\Fe (t) - \fz (t) \| \leq C_1(I) \eps,\;\;t \in I, \;\;\eps >0.
\]
\end{pro}
\begin{proof}
By Proposition \ref{ComSuppLoc} ii), the solution of the problem \eqref{EquNew3}, \eqref{EquNew4} has compact support, uniformly for $t$ in bounded intervals {\it i.e.,} $
\forall \;I \in \R_+,\exists \;R(I)$ such that 
\[
\supp \fz (t, \cdot, \cdot, \cdot) \subset \{ (s,x,p) \;:\; |x| \leq R(I),\;\;|p| \leq R (I)\},\;\;t \in [0,I].
\]
Under our hypotheses, the solution $\fz$ has the regularity 
\[
\fz, \nabla _{(t,x,p)} \fz, \nabla ^2 _{(t,x,p)} \fz \in L^\infty ([0,I] \times \R_s \times \R ^3 \times \R ^3),\;\;I \in \R_+.
\]
Recall that the model \eqref{Equ26} is equivalent to \eqref{Equ24}. By Proposition \ref{P2}, for any $t \in \R_+$, there is a unique function $\fo (t)$ of zero average such that 
\begin{equation}
\label{Equ42} \partial _t \fz + \frac{p}{m} \cdot \nabla _x \fz + ( e E(t,x) + \oc \tes{} \;^\perp p ) \cdot \nabla _p \fz + {\cal T}\fo (t) = 0.
\end{equation}
Since $\fz$ is smooth and has compact support (uniformly with respect to $t$ in bounded intervals) the following set
\[
\left \{ \nabla _{(t,x,p)} \left ( \partial _t \fz + \frac{p}{m} \cdot \nabla _x \fz + ( e E(t,x) + \oc \tes{} \;^\perp p ) \cdot \nabla _p \fz \right ) \;:\; t \in [0,I] \right \}
\]
is bounded in $\ltpltxp{}$. By Propositions \ref{TimDer}, \ref{RegH1} we deduce that 
\[
\{ \nabla _{(t,x,p)} \fo (t)\;:\;t \in [0,I]\}
\]
remains bounded in $\ltpltxp{}$. Combining \eqref{Equ40}, \eqref{Equ42} and the constraint \eqref{Equ22} yields
\begin{eqnarray}
&  &\left ( \partial _t + \frac{p}{m} \cdot \nabla _x + (eE(t,x) + \oc \tes{} \;^\perp p ) \cdot \nabla _p \right )\{\Fe - \fz - \eps \fo\} + \frac{1}{\eps} {\cal T} \{\Fe - \fz - \eps \fo\}\nonumber \\
& & = - \eps \left \{\partial _t \fo + \frac{p}{m} \cdot \nabla _x \fo + ( eE(t,x) + \oc \tes{} \;^\perp p ) \cdot \nabla _p \fo \right \} \nonumber
\end{eqnarray}
and after integration with respect to $(s,x,p)$ one gets
\[
\frac{d}{dt}\|\Fe - \fz - \eps \fo\| \leq \eps \left \|\partial _t \fo + \frac{p}{m} \cdot \nabla _x \fo + ( eE(t) + \oc \theta \;^\perp p ) \cdot \nabla _p \fo  \right \|.
\]
Taking into account that $\Fe, \fz$ satisfy the same initial condition we deduce that 
\[
\|\Fe (t) - \fz (t) - \eps \fo (t) \| \leq \eps \|\fo (0) \| + \eps \int _0 ^t \left \|\left ( \partial _t + \frac{p}{m} \cdot \nabla _x + ( eE(r) + \oc \theta \;^\perp p ) \cdot \nabla _p \right ) \fo  \right \|\;\mathrm{d}r 
\]
and finally 
\[
\|\Fe (t) - \fz (t) \| \leq C_1 (I) \eps,\;\;t \in I,\;\;\eps >0.
\]
\end{proof}
\begin{pro}
\label{UnifCompSuppF} Assume that the electric field is smooth $E \in L^1 _{\mathrm{loc}}(\R_+; W^{1,\infty}(\R^3))$ and that $\fin$ has compact support $\supp \fin \subset \{(x,p) : |x| \leq \Rin,|p| \leq \Rin\}$. Then there is a continuous nondecreasing function $\delta : \R_+ \to \R_+$ such that 
\[
\supp F ^\eps  (t, \cdot, \cdot, \cdot)  \subset \{ (s,x,p) \;:\; |p| \leq \delta (t),\;\frac{m |\oc|\;\|\theta \|_{L^\infty}}{2} |x| \leq \delta (t)\},\;t \in \R_+, \;\eps >0.
\] 
In particular for any $I \in \R_+$
\[
 \bigcup _{\eps >0, t \in [0,I]} \;\supp F ^\eps (t, \cdot, \cdot, \cdot)  \subset \{ (s,x,p) \;:\; |p| \leq \delta (I), \;\frac{m |\oc|\;\|\theta \|_{L^\infty}}{2} |x| \leq \delta (I)\}.
\]
\end{pro}
\begin{proof}
Let us consider a function $\xi \in C^1 (\R)$ satisfying
\[
0 \leq \xi \leq 1,\;\;\xi ^{\;\prime} \geq 0,\;\;\supp \xi = \R_+,\;\;z \xi ^{\;\prime} (z) \leq C \xi (z),\;\;z \in \R
\]
for some constant $C>0$. We denote by $h$ the function
\[
h(s,x,p) = \left | p - \frac{m \oc}{2} \tes{} \;^\perp x \right | + \frac{m |\oc|}{2} \|\theta \|_{L^\infty} \;|x|.
\]
Notice that $h$ depends only on the invariants $x$ and $q = p -  \frac{m \oc}{2} \tes{} \;^\perp x$ and therefore ${\cal T}h = 0$. By direct computations one gets
\[
\nabla _x \left | p - \frac{m \oc}{2} \tes{} \;^\perp x \right |  = \frac{m \oc}{2} \tes{} \frac{^\perp q}{|q|},\;\;\nabla _p \left | p - \frac{m \oc}{2} \tes{} \;^\perp x \right |  = \frac{q}{|q|}
\]
and therefore
\begin{eqnarray}
\left (  \frac{p}{m} \cdot \nabla _x + (e E + \oc \tes{} \;^\perp p ) \cdot \nabla _p \right ) \left | p - \frac{m \oc}{2} \tes{} \;^\perp x \right |  & = & e E(t) \cdot \frac{q}{|q|} + \frac{m \oc ^2 \theta ^2 (s) }{4} \;^{\perp \perp} x \cdot \frac{q}{|q|} \nonumber \\
& \leq & \|e E(t)\|_{\linf{}} + \frac{m \oc ^2 }{4} \|\theta \|_{L^\infty} ^2 |x|. \nonumber 
\end{eqnarray}
Similarly 
\begin{eqnarray}
\left (  \frac{p}{m} \cdot \nabla _x + (eE(t,x) + \oc \tes{} \;^\perp p ) \cdot \nabla _p \right ) \frac{m |\oc| \;\|\theta\|_{L^\infty}}{2}|x| & = &  \frac{p}{m}\cdot \frac{m |\oc| \;\|\theta\|_{L^\infty}}{2} \frac{x}{|x|} \nonumber \\
& \leq &  \frac{|\oc | \;\|\theta \|_{L^\infty}}{2} |q|. \nonumber 
\end{eqnarray}
Finally the function $h$ satisfies
\begin{eqnarray}
\left ( \partial _t + \frac{p}{m} \cdot \nabla _x + (eE(t,x) + \oc \tes{} \;^\perp p ) \cdot \nabla _p  + \frac{1}{\eps} {\cal T}\right )h 
%& \leq & \|e E (t) \|_{\linf{}} + \frac{m \oc ^2 }{4} \|\theta \|^2 _{L^\infty} |x| + \frac{|\oc | \;\|\theta \|_{L^\infty}}{2} |q| \nonumber \\
& \leq & \|e E (t) \|_{\linf{}} + \frac{|\oc | \;\|\theta \|_{L^\infty}}{2} h.\nonumber
\end{eqnarray}
Using now the weak formulation of \eqref{Equ40} with the test function $(t,s,x,p)  \to \xi (h(s,x,p) - \delta (t))$, with $\delta \in C^1 (\R_+)$ yields
\begin{eqnarray}
\label{EquNew61} & & \intsxp{\Fe (t) \xi (h (s,x,p) - \delta (t))} - \intsxp{\Fe (0) \xi (h (s,x,p) - \delta (0))} \nonumber \\
& \leq & \!\!\!\int _0 ^t \!\!\!\intsxp{\Fe (\tau) \xi ^{\;\prime} (h(s,x,p) - \delta (\tau)) \left \{ - \delta ^{\;\prime }(\tau) + \|e E (\tau) \|_{L^\infty} + \frac{|\oc | \;\|\theta \|_{L^\infty}}{2} h\right \}}\mathrm{d}\tau. \nonumber \\
& & 
\end{eqnarray}
Notice that 
\[
\supp \Fe (0, \cdot, \cdot, \cdot) \subset \{  (s,x,p)\;:\; |x| \leq \Rin, \;|p| \leq ( 1 + m |\oc | \; \|\theta \|_{L^\infty} ) \Rin\}
\]
and therefore $\delta _0 : = \sup \{h(s,x,p):(s,x,p) \in \cup _{\eps >0} \supp \Fe (0) \} < +\infty$. In this case we have
\begin{equation}
\label{EquNew62} \intsxp{\Fe (0, s, x, p)  \xi (h(s,x,p) - \delta _0)} = 0,\;\;\eps >0.
\end{equation}
We determine the function $\delta$ by solving
\[
\delta ^ {\;\prime} (\tau) = \|e E (\tau) \|_{L^\infty} + \frac{|\oc | \;\|\theta \|_{L^\infty}}{2} \delta (\tau),\;\;\tau \in \R_+
\]
with the initial condition $\delta (0) = \delta _0$. The right hand side of \eqref{EquNew61} becomes
\begin{eqnarray}
\label{EquNew63} 
& & \int _0 ^t \!\!\!\intsxp{\Fe (\tau) \xi ^{\;\prime} (h(s,x,p) - \delta (\tau))  \frac{|\oc | \;\|\theta \|_{L^\infty}}{2} ( h - \delta (\tau)) }\mathrm{d}\tau \nonumber \\
& \leq & C \;\frac{|\oc | \;\|\theta \|_{L^\infty}}{2}\int _0 ^t \!\!\!\intsxp{\Fe (\tau) \xi  (h(s,x,p) - \delta (\tau)) }\mathrm{d}\tau. 
\end{eqnarray}
Combining \eqref{EquNew61}, \eqref{EquNew62}, \eqref{EquNew63} implies 
\[
\intsxp{\Fe (t) \xi  (h - \delta (t)) }\leq C \;\frac{|\oc | \;\|\theta \|_{L^\infty}}{2}\int _0 ^t \!\!\!\intsxp{\Fe (\tau) \xi  (h - \delta (\tau)) }\mathrm{d}\tau
\]
and by Gronwall's lemma we deduce that 
\[
\intsxp{\Fe (t,s,x,p) \xi  (h(s,x,p) - \delta (t)) } = 0,\;\;t \in \R_+
\]
and therefore
\[
\supp \Fe (t) \subset \{  (s,x,p) \;:\; h (s,x,p) \leq \delta (t) \} \subset \{ (s,x,p) \;:\; |p| \leq \delta (t),\;\frac{m |\oc|\;\|\theta \|_{L^\infty}}{2} |x| \leq \delta (t)\}.
\]
\end{proof}
Once we have estimated the error between $\Fe$ and $\fz$, the asymptotic behaviour of $\fe (t,x,p) - \fz(t,\s, x, p)$ as $\eps \searrow 0$ follows by using the Sobolev inequality in Proposition \ref{P3}.
\begin{proof} (of Theorem \ref{MainResult1})
By Proposition \ref{P3} we have
\begin{eqnarray}
\|\fe (t,\cdot, \cdot) - \fz(t,\s, \cdot, \cdot) \|_{\ltxp{}}& = & \|\Fe (t, \s, \cdot, \cdot) - \fz(t,\s, \cdot, \cdot) \|_{\ltxp{}} \nonumber \\
& \leq & \|\Fe (t,\cdot, \cdot, \cdot) - \fz(t, \cdot, \cdot, \cdot) \|_{\lipltxp{}} \nonumber \\
& \leq & C(T) ( \|\Fe (t) - \fz (t) \| + \|{\cal T} \Fe (t) \|) \nonumber 
\end{eqnarray}
since ${\cal T}\fz (t) = 0$ for any $t \in \R_+$. Thanks to Proposition \ref{P4} we know that 
\[
\|\Fe (t) - \fz (t) \| \leq C_1(I) \eps,\;\;t \in [0,I],\;\;\eps >0
\]
and we are done if we can find a constant $C_2 (I)$ such that 
\[
\|{\cal T}\Fe (t) \| \leq C_2 (I) \eps,\;\;t \in [0,I],\;\;\eps >0
\]
since in that case we would obtain
\[
\|\fe(t,\cdot, \cdot) - \fz(t,\s, \cdot, \cdot) \|_{\ltxp{}} \leq C(T)(C_1 (I) + C_2 (I) ) \eps,\;\;t \in [0,I],\;\;\eps >0.
\]
Obviously, multiplying \eqref{Equ40} by $\Fe$ and integrating with respect to $(s,x,p)$ we control the $L^2$ norm of $\Fe (t, \cdot, \cdot, \cdot)$ uniformly in $t \in \R_+$ and $\eps >0$
\[
\|\Fe (t) \|^2 = \|\Fe (0) \|^2 = T \;\|\fin \|^2 _{\ltxp{}}.
\]
We intend to control the derivatives $\nabla _{(t,x,p)} \Fe$ as well, uniformly with respect to $\eps >0$. The idea is to use the derivations commuting with ${\cal T}$, introduced in Proposition \ref{FamInv}. Indeed, with the notation $a(t,s,x,p) = ( 0, \frac{p}{m}, e E(t,x)  + \oc \tes{} \;^\perp p)$ the equation \eqref{Equ40} becomes
\[
\partial _t \Fe + a \cdot \nabla _{(s,x,p)} \Fe + \frac{1}{\eps} {\cal T} \Fe = 0.
\] 
Applying the operator $c ^i \cdot \nabla _{(s,x,p)}$, $i \in \{1,...,6\}$ and taking into account that $c ^i \cdot \nabla _{(s,x,p)} {\cal T}\Fe = {\cal T} ( c ^i \cdot \nabla _{(s,x,p)} \Fe)$ one gets
\begin{equation}
\label{EquNew31}
\partial _t G ^\eps _i + a \cdot \nabla _{(s,x,p)} G^\eps _i  + [c^i,a] \cdot \nabla _{(s,x,p)} \Fe + \frac{1}{\eps } {\cal T} G^\eps _i = 0
\end{equation}
where $G^\eps _i = c ^i \cdot \nabla _{(s,x,p)} \Fe$ and $[c^i,a]$ are the Poisson brackets between the fields $c^i$, $i \in \{1,...,6\}$ and $a$. Multiplying \eqref{EquNew31} by $G ^\eps _i$ and integrating with respect to $(s,x,p)$ yield
\begin{eqnarray}
\frac{1}{2}\frac{d}{dt} \intsxp{|G ^\eps _i |^2 } & = & - \intsxp{G ^\eps _i \; [c^i,a] \cdot \nabla _{(s,x,p)} \Fe } \nonumber \\
& \leq & \|G ^\eps _i (t) \| \;\|[c^i,a] \cdot \nabla _{(s,x,p)} \Fe (t) \|,\;\;i \in \{1,...,6\} \nonumber 
\end{eqnarray}
or equivalently 
\begin{equation}
\label{EquNew32} 
\|G ^\eps _i (t) \| \leq \|G ^\eps _i (0) \| + \int _0 ^t \|[c^i,a] \cdot \nabla _{(s,x,p)} \Fe ( \tau) \|\;\mathrm{d}\tau,\;\;i \in \{1,...,6\}.
\end{equation}
It is easily seen that for any $i \in \{1,...,6\}$ the field $[c^i,a]$ has no component along $s$, since $c^i _s = a_s = 0$. Therefore $[c^i,a] \in \mathrm{span} \{c^1, ..., c^6\}$
\[
[c^i,a] = \sum _{j = 1} ^6 \gamma _{ij} (t,s,x,p) \;c^j,\;\;i \in \{1,...,6\}
\]
for some coefficients $\gamma _{ij} \in L^1 _{\mathrm{loc}} ( \R_+; L^\infty (\R \times \R^3 \times \R^3))$, $i,j \in \{1,...,6\}$. Actually we have
\[
[c^1,a ] = - \frac{\oc \tes{}}{2}c^2 + \left ( e \partial _{x_1} E_1 - \frac{m \oc ^2 }{4} \theta ^2 (s) \right ) c ^4 + e \partial _{x_1} E_2 c^5 + e \partial _{x_1} E_3 c ^6
\]
\[
[c^2,a ] =  \frac{\oc \tes{}}{2}c^1 +   e \partial _{x_2} E_1 c^4 + \left ( e \partial _{x_2} E_2 - \frac{m \oc ^2 }{4} \theta ^2 (s) \right ) c ^5 + e \partial _{x_2} E_3 c ^6
\]
\[
[c^3,a ] = \sum _{j = 1} ^3 e \partial _{x_3} E_j \;c ^{j+3},\;\;[c^4,a ] = \frac{c^1 }{m} - \frac{\oc }{2} \tes{} \;c^5,\;\;[c^5,a ] = \frac{c^2 }{m} + \frac{\oc }{2} \tes{} \;c^4,\;\;
[c^6,a ] = \frac{c^3}{m}.
\]
It follows that $[c^i, a] \cdot \nabla _{(s,x,p)} \Fe = \sum _{j = 1} ^6 \gamma _{ij} G ^\eps _j$ and \eqref{EquNew32} yields
\[
\|G ^\eps _i (t) \| \leq \| G ^\eps _i (0) \| + \int _0 ^t \sum _{j = 1} ^ 6 \|\gamma _{ij} (\tau) \|_{L^\infty} \|G ^\eps _j (\tau) \|\;\mathrm{d}\tau.
\]
We deduce that 
\[
\sum _{i = 1} ^6 \|G ^\eps _i (t) \| \leq \sum _{i =1 } ^ 6  \|G ^\eps _i (0) \| + \int _0 ^t \gamma (\tau)  \sum _{j = 1} ^6  \|G ^\eps _j (\tau) \| \mathrm{d} \tau
\]
with $\gamma (\tau) = \max _{j \in \{1,...,6\}} \sum _{i = 1} ^6 \|\gamma _{ij} (\tau) \|_{L^\infty}$ and by Gronwall's lemma we obtain that
\[
\{ \nabla _{(x,p)} \Fe (t) \;:\;t \in [0,I],\;\;\eps >0\}
\]
remains bounded in $\ltpltxp{}$. By Proposition \ref{UnifCompSuppF} we know that $\cup _{\eps>0, t \in [0,I]} \supp \Fe (t)$ remains into a compact set of $\R _s /T\Z \times \R ^3 \times \R^3$ and clearly there is a constant $C_3 (I)$ such that for any $t \in [0,I]$
\begin{eqnarray}
\label{EquNew71}
\sup _{\eps >0} \left \| \frac{p}{m} \cdot \nabla _x \Fe (t) + ( e E + \oc \tes{} \;^\perp p ) \cdot \nabla _p \Fe    \right \|  & \leq & C_3 (I) ( 1 + \|E(t) \|_{\linf{}})  \\
& \leq & C_4 (I) ( 1 + \|E(0) \|_{\linf{}} + \|\partial _t E \|_{L^1 ([0,t];\linf{})}).\nonumber
\end{eqnarray}
It remains to estimate the time derivative $\partial _t \Fe$. As before we write
\[
\partial _t (\partial _t \Fe) + a \cdot \nabla _{(s,x,p)} ( \partial _t \Fe ) + e \partial _t E \cdot \nabla _p \Fe + \frac{1}{\eps} {\cal T} (\partial _t \Fe ) = 0
\]
implying that 
\begin{eqnarray}
\frac{1}{2}\frac{d}{dt} \|\partial _t \Fe \|^2 & = &  - e \intsxp{\partial _t E (t) \cdot \nabla _p \Fe (t) \;\partial _t \Fe (t)} \nonumber \\
& \leq &  \|e \partial _t E(t) \cdot \nabla _p \Fe (t) \| \; \|\partial _t \Fe (t) \|. \nonumber 
\end{eqnarray}
We deduce that 
\[
\|\partial _t \Fe (t) \|\leq \|\partial _t \Fe (0) \| + \int _0 ^t \|e \partial _t E (\tau) \|_{L^\infty} \|\nabla _p \Fe (\tau) \|\;\mathrm{d}\tau.
\]
The family of time derivatives $\{\partial _t \Fe (t): t \in [0,I],\eps >0\}$ remains bounded in $\ltpltxp{}$ iff $\{\|\partial _t \Fe (0) \| :\eps >0\}$ remains bounded in $\ltpltxp{}$. Notice that ${\cal T}\Fe (0) = 0$ and therefore
\begin{eqnarray}
\label{EquNew72} \sup _{\eps >0} \|\partial _t \Fe (0) \| & = & \sup _{\eps >0} \| - a \cdot \nabla _{(s,x,p)} \Fe (0) \|  \\
& = & \sup _{\eps >0} \left \| - a \cdot \nabla _{(s,x,p)} \fin \left (x, p - \frac{m \oc}{2} \tes{} \;^\perp x + \frac{m \oc}{2} \theta (0) \;^\perp x  \right )   \right \| < +\infty.\nonumber
\end{eqnarray}
Combining \eqref{EquNew71}, \eqref{EquNew72} we deduce that there is a constant $C_2 (I)$ such that 
\[
\sup _{\eps >0, t \in [0,I]} \left \| \left ( \partial _t + \frac{p}{m} \cdot \nabla _x + (e E(t) + \oc \tes{} \;^\perp p ) \cdot \nabla _p \right ) \Fe (t)   \right \| \leq \eps C_2 (I)
\]
saying that $\sup _{\eps >0, t \in [0,I]} \|{\cal T}\Fe (t) \| \leq \eps C_2 (I)$.
\end{proof}

\section{Three dimensional setting}
\label{3DSetting}

In this section we study the particle dynamics under fast oscillating three dimensional magnetic fields
\[
\Be (t,x) = \theta (\s) B(x) b(x),\;\;\Div (Bb) = 0
\]
for some scalar positive function $B(x)$ and some field of unitary vectors $b(x) \in \R^3$. The analysis is completely analogous to that for fast oscillating homogeneous magnetic fields previously discussed. Therefore we only focus on the formal derivation of the limit model. By Gauss's magnetic law $\Div \Be = 0$ we can write $Bb = \Curl A$, $\Div A = 0$ and by Faraday's law $\partial _t \Be + \Curl \Ee = 0$ we deduce that the rotational part, $\Curl \psi$, of the electric field $\Ee = - \nabla _x \phi  + \Curl \psi$ is given by
\[
\Curl \psi = - \frac{1}{\eps} \tept{} A(x).
\]
The Vlasov equation becomes, with the notations $E = - \nabla _x \phi$, $ \oc (x) = \frac{eB(x)}{m}$
\begin{equation}
\label{EquNew73} \partial _t \fe + \frac{p}{m} \cdot \nabla _x \fe + \left ( e E(t,x) - \frac{e \tept{}}{\eps} A(x) + \oc (x) \tet{} p \wedge b(x) \right ) \cdot \nabla _p \fe = 0.
\end{equation}
We prescribe the initial distribution
\begin{equation}
\label{EquNew84} \fe (0,x,p) = \fin (x,p),\;\;(x,p) \in \R ^3 \times \R ^3.
\end{equation}
Plugging the Hilbert expansion \eqref{Equ14} into \eqref{EquNew73} yields
\begin{equation}
\label{EquNew74} \partial _s \fz - e \;\teps{} A(x) \cdot \nabla _p \fz = 0
\end{equation}
at the lowest order $\eps ^{-1}$ and
\begin{equation}
\label{EquNew75} \partial _t \fz + \frac{p}{m} \cdot \nabla _x \fz + \left ( e E(t,x) + \oc (x) \tes{} p \wedge b(x) \right ) \cdot \nabla _p \fz + \partial _s \fo - e \;\teps{} A(x) \cdot \nabla _p \fo = 0
\end{equation}
at the next order $\eps ^0$. As before, the point is how to eliminate the first order correction $f^1$ appearing in \eqref{EquNew75}, based on the constraint \eqref{EquNew74}. We introduce the operator
\[
{\cal T}_1 u = \mathrm{div}_{(s,p)}\{ u (1,- e \;\teps{}A(x))\}
\]
with domain
\[
D({\cal T}_1) = \{u   \in \ltpltxp{}:\mathrm{div}_{(s,p)} \left \{u \left ( 1, - e \;\teps{} A(x) \right ) \right \} \in \ltpltxp{}\}.
\]
The characteristics $(S,X,P)(\tau;s,x,p)$ of the first order differential operator ${\cal T}_1$ are given by
\begin{equation}
\label{EquNew76} S(\tau;s,x,p) = s + \tau,\;\;X(\tau;s,x,p) = x,\;\;P(\tau;s,x,p) = p + e ( \tes{} - \theta ( s + \tau)) A(x).
\end{equation}
Notice that a complete family of functional independent invariants is given by $\{x, p + e \tes{} A(x)\}$ and therefore the constraint \eqref{EquNew74} becomes
\[
\exists \;\gz = \gz(t,x,q)\;:\; \fz (t,s,x,p) = \gz (t, x, q = p + e \theta (s) A(x)).
\]
In particular $\fz (t, t/\eps, x, p)$ is fast oscillating through the periodic profile $\theta( s = t/\eps)$ and therefore we expect that $\fz (t,s,x,p)$ is the two-scale limit of $\fe (t,x,p)$ when $\eps \searrow 0 $. The average operator $\ave{\cdot}_1$ along the characteristic flow \eqref{EquNew76} is given by
\begin{eqnarray}
\ave{u}_1 (s,x,p) & = &  \frac{1}{T}\int _0 ^T u (S(\tau;s,x,p), X(\tau;s,x,p), P(\tau;s,x,p)) \;\mathrm{d}\tau  \\
& = &  \frac{1}{T}\int _0 ^T u \left (s+ \tau, x, p + e ( \theta (s ) - \theta ( s + \tau)) A( x) \right )\;\mathrm{d}\tau \nonumber \\
& = & \frac{1}{T} \int _0 ^T u \left ( \tau, x, p +e\; \tes{} A( x) - e \; \theta (\tau) A(x)    \right ) \;\mathrm{d}\tau\nonumber 
\end{eqnarray}
for any function $u \in \ltpltxp{}$. The dynamics for $\fz$ is obtained by eliminating $\fo$ in \eqref{EquNew75} taking into account that the functions in the range of ${\cal T}_1$ are zero average. We have
\[
\partial _t \fz + \frac{p}{m} \cdot \nabla _x \fz + \left ( e E(t,x) + \oc (x) \tes{} p \wedge b(x) \right ) \cdot \nabla _p \fz  = - {\cal T}_1 \fo \in \ran {\cal T}_1 = \ker \ave{\cdot}_1
\]
and therefore \eqref{EquNew75} is equivalent to
\[
\ave{\partial _t \fz + \frac{p}{m} \cdot \nabla _x \fz + \left ( e E(t,x) + \oc (x) \tes{} p \wedge b(x) \right ) \cdot \nabla _p \fz }_1 = 0.
\]
We need to average the derivatives with respect to $(t,x,p)$ of the density $\fz$, under the constraint \eqref{EquNew74}. Clearly we have $\ave{\partial _t \fz}_1 = \partial _t \ave{\fz}_1 = \partial _t \fz$.
\begin{lemma}
\label{DerXP3D} Assume that $f(s,x,p) = g\left ( x,q = p + e\;\tes{} A(x)\right )$ is smooth. Then we have
\begin{eqnarray}
& & \ave{\frac{p}{m}\cdot \nabla _x f }_1  =  \frac{q- e\ave{\theta}A(x)}{m}   \cdot \nabla _x g + \frac{e}{m} \partial _x A  ( \ave{\theta} q - e \ave{\theta ^2} A) \cdot \nabla _q g \nonumber \\
& = & \frac{p + e ( \tes{} - \ave{\theta})A(x)}{m} \cdot \nabla _x f + \frac{e}{m} \partial _x A  \left ( ( \ave{\theta} - \theta ) p + e ( 2 \theta \ave{\theta} - \theta ^2 - \ave{\theta ^2} ) A(x) \right )\cdot \nabla _p f      \nonumber 
\end{eqnarray}
and
\begin{eqnarray}
 \ave{(eE(x) + \oc \theta \;p \wedge b ) \cdot \nabla _p f }_1 & = & \left [eE  + \oc ( \ave{\theta} q \wedge b - e \ave{\theta ^2} A(x) \wedge b ) \right ] \cdot \nabla _q g
\nonumber \\
& = & \left [e E + \oc \ave{\theta} p\wedge b  + \oc e (\theta \ave{\theta}-\ave{\theta ^2})A( x) \wedge b     \right ] \cdot \nabla _p f. \nonumber  
\end{eqnarray}
\end{lemma}
\begin{proof}
We have
\[
\nabla _x f = \nabla _x g + e \;\tes{} \;^t \partial _x A \nabla _q g,\;\;\nabla _p f = \nabla _q g.
\]
Since $\nabla _{(x,q)} g$ are constant along the flow \eqref{EquNew76} we can write
\[
\ave{\frac{p}{m}\cdot \nabla _x f }_1  = \frac{\ave{p}_1}{m}\cdot \nabla _x g + \frac{e}{m} \ave{p \;\theta}_1 \cdot \;^t \partial _x A \nabla _q g.
\]
It is easily seen by the definition of the average operator $\ave{\cdot}_1$ that 
\[
\ave{p}_1 = p + e \;\tes{} A(x) - e \ave{\theta} A(x)
\]
and
\[
\ave{p \;\theta}_1  = ( p + e \;\tes{} A(x) ) \ave{\theta} - e \ave{\theta ^2} A(x)
\]
implying that 
\[
\ave{\frac{p}{m}\cdot \nabla _x f }_1  =  \frac{q- e\ave{\theta}A(x)}{m}   \cdot \nabla _x g + \frac{e}{m} \partial _x A  ( \ave{\theta} q - e \ave{\theta ^2} A(x)) \cdot \nabla _q g. 
\]
Similarly one gets
\[
\ave{(eE(x) + \oc \theta \;p \wedge b ) \cdot \nabla _p f }_1  =  \left [eE  + \oc ( \ave{\theta} q \wedge b - e \ave{\theta ^2} A(x) \wedge b ) \right ] \cdot \nabla _q g.
\]
\end{proof}
Combining the previous computations and using the identities
\[
\partial _x A \;A + A \wedge \Curl A = \;^t \partial _x A\;A,\;\;\partial _x A \;q + q \wedge \Curl A = \;^t \partial _x A \;q
\]
yield the transport equation in the space phase $(x,q)$
\begin{equation}
\label{EquNew77} \partial _t \gz + \frac{q - e \;\ave{\theta}A(x)}{m} \cdot \nabla _x \gz + \left ( eE + \frac{e}{m} \ave{\theta} \;^t \partial _x A \;q - \frac{e^2}{m} \ave{\theta ^2} \;^t \partial _x A \;A \right ) \cdot \nabla _q \gz = 0
\end{equation}
where $\fz (t,s,x,p) = \gz (t,x,q = p + e \tes{}A(x))$, since ${\cal T}_1 \fz (t) = 0$ for any $t \in \R_+$. The transport equation in the phase space $(x,p)$ becomes
\begin{eqnarray}
\label{EquNew78} 
\partial _t \fz & + & \frac{p + e \;(\tes{} - \ave{\theta})A(x)}{m} \cdot \nabla _x \fz + \left [ eE + \frac{e}{m} ( \ave{\theta} \;^t \partial _x A  - \theta \partial _x A ) p \right. \nonumber \\
&+& \left. \frac{e^2}{m} ( \theta \ave{\theta} - \ave{\theta ^2} ) \;^t \partial _x A A  + \frac{e^2}{m} ( \ave{\theta} - \theta) \;\theta \;\partial _x A A \right ]\cdot \nabla _p \fz = 0.
\end{eqnarray}
We supplement these transport equations by the initial conditions
\begin{equation}
\label{EquNew79} \gz (0,x,q) = \fin (x, q - e\;\theta (0)A(x))
\end{equation}
\begin{equation}
\label{EquNew80} \fz (0,s,x,p) = \fin (x, p+ e\;(\tes{} - \theta (0))A(x)).
\end{equation}
Following the lines in Sections \ref{LimMod}, \ref{AsyBeh} we can prove weak and strong convergence results, which justify the Hilbert expansion in \eqref{Equ14}. In the weak framework we obtain
\begin{thm}
Assume that $E \in L^1 _{\mathrm{loc}} ( \R_+; L^\infty (\R^3))$, $A \in L^1 _{\mathrm{loc}} ( \R_+; W^{1,\infty}(\R ^3)) ^3$, $\fin \in L^2 (\R ^3 \times \R ^3)$. For any $ \eps >0$ let $\fe \in L^\infty (\R_+; L^2 (\R ^3 \times \R ^3))$ be a weak solution of \eqref{EquNew73}, \eqref{EquNew84}. Then there is a sequence $\eps _n \searrow 0$ such that $(f^{\eps _n})_n$ two-scale converges towards a weak solution of \eqref{EquNew78}, \eqref{EquNew80}.
\end{thm}

%\footnotesize{
%\scriptsize{

%}
\end{document}